\documentclass[reqno,10pt]{amsart}

\usepackage{a4wide}
\usepackage{mathrsfs}
\usepackage{mathtools}
\usepackage{amsmath}
\usepackage{amssymb}
\usepackage{bbm}
\usepackage{esint}
\numberwithin{equation}{section}
\usepackage[colorlinks,citecolor=green,linkcolor=red]{hyperref}
\usepackage{tikz-cd}

\usepackage[latin1]{inputenc}

\newcommand{\N}{\mathbb{N}}

\newcommand{\R}{\mathbb{R}}
\newcommand{\B}{\mathbb{B}}
\newcommand{\sfd}{{\sf d}}
\renewcommand{\d}{{\mathrm d}}

\newcommand{\X}{{\rm X}}

\newcommand{\mm}{\mathfrak{m}}
\newcommand{\1}{\mathbbm 1}
\newcommand{\q}{\mathfrak{q}}

\newcommand{\fr}{\penalty-20\null\hfill\(\blacksquare\)}

\newcommand{\restr}[1]{\lower3pt\hbox{$|_{#1}$}}

\newcommand{\limi}{\varliminf}
\newcommand{\lims}{\varlimsup}

\newtheorem{theorem}{Theorem}[section]

\newtheorem{lemma}[theorem]{Lemma}
\newtheorem{proposition}[theorem]{Proposition}
\newtheorem{definition}[theorem]{Definition}
\newtheorem{remark}[theorem]{Remark}
\newtheorem{example}[theorem]{Example}

\title[On the integration of Banach modules and its applications to ${\sf RCD}$ spaces]{On the integration of Banach modules and its \\ applications to vector calculus on ${\sf RCD}$ spaces}

\author[Emanuele Caputo]{Emanuele Caputo}
\address[Emanuele Caputo]{Department of Mathematics and Statistics, P.O.\ Box 35 (MaD), FI-40014 University of Jyvaskyla, Finland}
\email{emanuele.e.caputo@jyu.fi}

\author[Milica Lu\v{c}i\'{c}]{Milica Lu\v{c}i\'{c}}
\address[Milica Lu\v{c}i\'{c}]{Department of Mathematics and Informatics, Faculty of Sciences, University of Novi Sad, Trg D.\ Obradovi\'{c}a 4, 21000 Novi Sad, Serbia}
\email{milica.lucic@dmi.uns.ac.rs}

\author[Enrico Pasqualetto]{Enrico Pasqualetto}
\address[Enrico Pasqualetto]{Department of Mathematics and Statistics, P.O.\ Box 35 (MaD), FI-40014 University of Jyvaskyla, Finland}
\email{enrico.e.pasqualetto@jyu.fi}

\author[Ivana Vojnovi\'{c}]{Ivana Vojnovi\'{c}}
\address[Ivana Vojnovi\'{c}]{Department of Mathematics and Informatics, Faculty of Sciences, University of Novi Sad, Trg D.\ Obradovi\'{c}a 4, 21000 Novi Sad, Serbia}
\email{ivana.vojnovic@dmi.uns.ac.rs}
\date{\today}
\begin{document}
\maketitle
\begin{abstract}
A finite-dimensional ${\sf RCD}$ space can be foliated into sufficiently regular leaves, where a differential calculus can be performed. Two important examples are given by the measure-theoretic boundary of the superlevel set of a function of
bounded variation and the needle decomposition associated to a Lipschitz function. The aim of this paper is to connect the vector calculus on the lower dimensional leaves with the one on the base space.
In order to achieve this goal, we develop a general theory of integration of $L^0$-Banach $L^0$-modules of independent interest. Roughly speaking, we study how to `patch together' vector fields defined on the leaves that
are measurable with respect to the foliation parameter.
\end{abstract}
\tableofcontents
\section{Introduction}
In the setting of metric measure spaces, the theory of \(L^0\)-normed \(L^0\)-modules was introduced by Gigli in \cite{Gigli14} and later refined in \cite{Gigli17}. The goal was to develop a vector calculus
in this nonsmooth framework, which could provide effective notions of measurable \(1\)-forms and vector fields. Indeed, one of the main objects of study in \cite{Gigli14,Gigli17} is the \emph{cotangent module}
\(L^0(T^*\X)\) over a metric measure space \((\X,\sfd,\mm)\), which is built upon the Sobolev space \(W^{1,2}(\X)\) and comes naturally with a differential operator \(\d\colon W^{1,2}(\X)\to L^0(T^*\X)\).
By duality, one then defines the \emph{tangent module} \(L^0(T\X)\). Both \(L^0(T^*\X)\) and \(L^0(T\X)\) are \emph{\(L^0(\mm)\)-Banach \(L^0(\mm)\)-modules}, see Definition \ref{def:Banach_module}. The study
of tangent and cotangent modules had remarkable analytic and geometric applications in the field of analysis on metric measure spaces, especially in the fast-developing theory of metric measure spaces verifying
synthetic Riemannian Ricci curvature lower bounds, the so-called \({\sf RCD}\) spaces (see the survey \cite{AmbICM} and the references therein indicated). However, the language of normed modules is used in other
research fields as well (see e.g.\ \cite{HLR91,Guo-2011}). Due to this reason, the technical machinery we develop in this paper is for \(L^0\)-Banach \(L^0\)-modules over a \(\sigma\)-finite measure space.
\medskip

Let us consider a \emph{disjoint measure-valued map} \(q\mapsto\mu_q\) on a measurable space \((\X,\Sigma)\). Roughly speaking, the measures \(\mu_q\)
are concentrated on pairwise disjoint measurable subsets of \(\X\) and vary in a measurable way with respect to the parameter \(q\in Q\), where \((Q,\mathcal Q,\mathfrak q)\) is a given measure space;
see Definition \ref{def:measure-valuedmap}. In Definition \ref{def:Banach_module_bundle} we propose a notion of \emph{Banach module bundle} \((\{\mathscr M_q\}_{q\in Q},\Theta,\Phi)\) consistent with
\(q\mapsto\mu_q\). Shortly said, each \(\mathscr M_q\) is an \(L^0(\mu_q)\)-Banach \(L^0(\mu_q)\)-module, while \(\Theta\) and \(\Phi\) are families of \emph{test functions} and \emph{test sections},
respectively, which are used to declare which sections \(Q\ni q\mapsto v_q\in\mathscr M_q\) are measurable. The space of all measurable sections of \((\{\mathscr M_q\}_{q\in Q},\Theta,\Phi)\) is then
denoted by \(\int\mathscr M_q\,\d\mathfrak q(q)\) and called the \emph{integral} of \((\{\mathscr M_q\}_{q\in Q},\Theta,\Phi)\); see Definition \ref{def:integral_of_bundle}. The space
\(\int\mathscr M_q\,\d\mathfrak q(q)\) has a structure of \(L^0(\mu)\)-Banach \(L^0(\mu)\)-module, where we define \(\mu\coloneqq\int\mu_q\,\d\mathfrak q(q)\).
Conversely, it is also possible to \emph{disintegrate} a given countably-generated \(L^0(\mu)\)-Banach
\(L^0(\mu)\)-module \(\mathscr M\), thus obtaining a Banach module bundle whose integral is the original
space \(\mathscr M\); this claim is proved in Theorem \ref{thm:disintegration_countably_generated} (and see also Theorem \ref{thm:tech_quotient}).
\medskip

We now discuss the applications of the above machinery to the vector calculus on metric measure spaces verifying lower Ricci bounds. Let \((\X,\sfd,\mm)\) be an \({\sf RCD}(K,N)\) space,
where \(K\in\R\) and \(N\in[1,\infty)\). We focus on two different classes of foliations (of a portion) of \(\X\): the codimension-one \emph{level sets} of a function of bounded variation
(BV function, for short) and, somehow dually, the one-dimensional \emph{needle decomposition} induced by a Lipschitz function. More in details:
\begin{itemize}
\item[\(\rm a)\)] The \emph{Fleming--Rishel coarea formula} proved by Miranda Jr.\ in \cite{MIRANDA2003} states that, given a BV function \(f\in BV(\X)\), the total variation measure \(|Df|\) coincides
with the `superposition' \(\int_\R P(\{f>t\},\cdot)\,\d t\) of the perimeter measures of the superlevel sets \(\{f>t\}\); see Theorem \ref{thm:coarea}. The fact that the level sets \(\{f>t\}\)
can be regarded as codimension-one objects is corroborated by the results of \cite{Amb01} (see also \cite{BPSGaussGreen}). Moreover, as shown in \cite{debin2019quasicontinuous}, the regularity
of \(\sf RCD\) spaces guarantees that (sufficiently many) vector fields can be `traced' over the essential boundary of a set of finite perimeter \(E\), thus obtaining the space \(L^0_{P(E,\cdot)}(T\X)\).
In particular, one can define the \emph{unit normal} \(\nu_E\in L^0_{P(E,\cdot)}(T\X)\) to \(E\), as proved in \cite{bru2019rectifiability}. More generally, one can construct the space
of \(|Df|\)-a.e.\ defined vector fields \(L^0_{|Df|}(T\X)\) whenever \(f\in BV(\X)\), see \cite{BGBV}. With this said, in Theorem \ref{thm:disintegration_L0Df} we will prove that
\begin{equation}\label{eq:intro_BV}
L^0_{|Df|}(T\X)\cong\int L^0_{P(\{f>t\},\cdot)}(T\X)\,\d t\quad\text{ for every }f\in BV(\X),
\end{equation}
once the family \(\R\ni t\mapsto L^0_{P(\{f>t\},\cdot)}(T\X)\) is made into a Banach module bundle, in a suitable way. To be precise, \eqref{eq:intro_BV} is not entirely correct as it is written,
since the measure-valued map \(t\mapsto P(\{f>t\},\cdot)\) is not necessarily disjoint due to the possible presence of the jump part of \(|Df|\); this technical subtlety will be dealt with in
Proposition \ref{prop:coarea_disj}. Notice that one can regard \eqref{eq:intro_BV} as a Fubini-type theorem for measurable vector fields. Similar results, but without integrals of \(L^0\)-Banach \(L^0\)-modules, previously appeared in \cite{AntonelliPasqualettoPozzetta21} (for the squared distance functions from a point) and in \cite{BGBV} (for arbitrary BV functions).
\item[\(\rm b)\)] Klartag's localisation technique \cite{Klartag2017} has been generalised from the Riemannian setting to metric measure spaces verifying synthetic lower Ricci bounds, see the survey
\cite{Cavalletti2017} and the references therein. Given a \(1\)-Lipschitz function \(u\colon\X\to\R\), the so-called \emph{transport set} \(\mathcal T_u\) can be partitioned (up to \(\mm\)-null sets)
into the images \(\X_\alpha\) of suitable geodesics \(\gamma_\alpha\), which are the gradient flow lines of \(-\nabla u\) in a suitable sense. The restriction \(\mm|_{\mathcal T_u}\)
of the reference measure \(\mm\) to \(\mathcal T_u\) can be disintegrated as \(\mm|_{\mathcal T_u}=\int\mm_\alpha\,\d\mathfrak q(\alpha)\), where each \(\mm_\alpha=h_\alpha\mathcal H^1\) is a measure
concentrated on \(\X_\alpha\). As proved in \cite{CavMon15}, the `needles' \((\X_\alpha,\sfd,h_\alpha\mathcal H^1)\) are \({\sf RCD}(K,N)\) spaces. The intuition suggests that the tangent module
\(L^0(T\X_\alpha)\) is one-dimensional and that a Lipschitz function \(f\colon\X\to\R\) induces a vector field \(\partial_\alpha f\in L^0(T\X_\alpha)\), which is obtained by differentiating
\(f\circ\gamma_\alpha\). This is made precise by Theorem \ref{thm:disint_needles}, which says that
\begin{equation}\label{eq:intro_needles}
\int L^0(T\X_\alpha)\,\d\mathfrak q(\alpha)\cong\langle\nabla u\rangle_{\mathcal T_u},
\end{equation}
where \(\langle\nabla u\rangle_{\mathcal T_u}\subseteq L^0(T\X)\) is the \(L^0(\mm|_{\mathcal T_u})\)-Banach \(L^0(\mm|_{\mathcal T_u})\)-module generated by \(\nabla u\), and for any Lipschitz function \(f\colon\X\to\R\) the superposition \(\int\partial_\alpha f\,\d\mathfrak q(\alpha)\in\int L^0(T\X_\alpha)\,\d\mathfrak q(\alpha)\) corresponds,
via the identification in \eqref{eq:intro_needles}, to the element \(-\langle\nabla f,\nabla u\rangle\nabla u\in\langle\nabla u\rangle_{\mathcal T_u}\).
\end{itemize}
Finally, in Theorem \ref{thm:unit_normal_vs_needles} we combine the two different foliations of \(\X\) we described above. Assuming \(\mm(\X)<+\infty\) for simplicity and fixed a \(1\)-Lipschitz
function \(u\colon\X\to\R\) with \(\mm(\X\setminus\mathcal T_u)=0\), we know from Theorems \ref{thm:disintegration_L0Df} and \ref{thm:disint_needles} that \(\int P(\{u>t\},\cdot)\,\d t\cong L^0(T\X)\)
and \(\int L^0(T\X_\alpha)\,\d\mathfrak q(\alpha)\cong\langle\nabla u\rangle_\X\), respectively. Via these identifications, we can prove that \(\int\partial_\alpha u\,\d\mathfrak q(\alpha)\)
corresponds to \(-\int\nu_{\{u>t\}}\,\d t\).
\subsection*{Acknowledgements}
E.C.\ acknowledges the support from the Academy of Finland, grants no.\ 314789 and 321896. M.L.\ and I.V.\ gratefully acknowledge the financial support of the Ministry of Science, Technological Development and Innovation of the Republic of Serbia (Grant No.\ 451-03-47/2023-01/200125).
\section{Preliminaries}
Throughout the whole paper, our standing convention (unless otherwise specified) is that
\[
\textbf{all the measures under consideration are }\boldsymbol\sigma\textbf{-finite.}
\]
Moreover, in this paper a ring \(R\) is not assumed to have a multiplicative identity. However, when the identity \(1_R\) exists, then we require that each \(R\)-module \(M\) verifies
\(1_R\cdot v=v\) for every \(v\in M\).
\medskip

Given a collection of non-empty sets \(A_\star=\{A_q\}_{q\in Q}\), we define the space of its sections as
\[
\mathscr S(A_\star)\coloneqq\prod_{q\in Q}A_q.
\]
In the case where each \(A_q\) is a ring (resp.\ a vector space), the product space \(\mathscr S(A_\star)\)
is a ring (resp.\ a vector space) with respect to the componentwise operations.
\begin{remark}\label{rmk:prod_mod}{\rm
Let \(R_\star=\{R_q\}_{q\in Q}\) be a collection of commutative rings. For any \(q\in Q\), let \(M_q\) be a module over \(R_q\).
Then \(\mathscr S(M_\star)\) has a natural structure of module over the ring \(\mathscr S(R_\star)\).
\fr}\end{remark}

Let \((Q,\mathcal Q,\mathfrak q)\) be a measure space. Let \(A_\star=\{A_q\}_{q\in Q}\) be rings (resp.\ vector spaces). Then
\[
N_{\mathfrak q}\coloneqq\Big\{v\in\mathscr S(A_\star)\;\Big|\;\{q\in Q\,:\,v_q\neq 0\}\subseteq N\text{ for some }N\in\mathcal Q\text{ with }\mathfrak q(N)=0\Big\}
\]
is a two-sided ideal (resp.\ a vector subspace) of \(\mathscr S(A_\star)\). Therefore, the quotient space
\[
\mathscr S_{\mathfrak q}(A_\star)\coloneqq\mathscr S(A_\star)/N_{\mathfrak q}
\]
is a subring (resp.\ a vector subspace). We denote by \(\pi_{\mathfrak q}\colon\mathscr S(A_\star)\to\mathscr S_{\mathfrak q}(A_\star)\) the quotient map.

According to this notation, we denote by \(\mathscr S_{\mathfrak q}(\R)\) the set of all the (possibly non-measurable)
functions from \(Q\) to \(\R\), quotiented up to \(\mathfrak q\)-a.e.\ equality.
\begin{remark}{\rm
Let \((Q,\mathcal Q,\mathfrak q)\) be a measure space and let \(R_\star\), \(M_\star\) be as in Remark \ref{rmk:prod_mod}. Then the vector space
\(\mathscr S_{\mathfrak q}(M_\star)\) inherits a natural structure of module over the ring \(\mathscr S_{\mathfrak q}(R_\star)\).
\fr}\end{remark}
\subsection{Integration and disintegration of measures}
Let \((\X,\Sigma,\mu)\) be a measure space. Then we denote by \(L^0(\mu)\) the space of all real-valued measurable functions on \(\X\), quotiented up to \(\mu\)-a.e.\ identity.
The space \(L^0(\mu)\) is a Riesz space, where we declare that \(f,g\in L^0(\mu)\) satisfy \(f\leq g\) if and only if \(f(x)\leq g(x)\) for \(\mu\)-a.e.\ \(x\in\X\). Moreover,
\(L^0(\mu)\) is Dedekind complete, which means that every order-bounded set \(\{f_i\}_{i\in I}\subseteq L^0(\mu)\) admits both a supremum \(\bigvee_{i\in I}f_i\in L^0(\mu)\)
and an infimum \(\bigwedge_{i\in I}f_i\in L^0(\mu)\). We also denote by \(\mathcal L^0(\Sigma)\) the space of all real-valued measurable functions on \(\X\), while \(\pi_\mu\)
(or \([\cdot]_\mu\)) stands for the canonical projection map \(\mathcal L^0(\Sigma)\to L^0(\mu)\). Given any \(E\subseteq\X\), we denote by \(\1_E\) its
characteristic function. If \(E\in\Sigma\), then we denote \(\1^\mu_E\coloneqq[\1_E]_\mu\).
\medskip

The space \(L^0(\mu)\) becomes a complete topological vector space and topological ring if endowed with the topology
that is induced by the following distance:
\[
\sfd_{L^0(\mu)}(f,g)\coloneqq\int|f-g|\wedge 1\,\d\tilde\mu\quad\text{ for every }f,g\in L^0(\mu),
\]
where \(\tilde\mu\geq 0\) is some finite measure on \((\X,\Sigma)\) with \(\mu\ll\tilde\mu\leq\mu\).
When \(\mu\) is finite, we take \(\tilde\mu=\mu\).
\medskip

We introduce a notion of disjoint measure-valued map, which we will use throughout the paper:
\begin{definition}[Measure-valued map]\label{def:measure-valuedmap}
Let \((Q,\mathcal Q,\mathfrak q)\) be a measure space and  \((\X,\Sigma)\) a measurable space. Let \(\mu_\star=\{\mu_q\}_{q\in Q}\) be a collection of measures on \((\X,\Sigma)\).
Then we say that \(q\mapsto\mu_q\) is a \textbf{(measurable) measure-valued map} from \(Q\) to \(\X\) provided the following conditions hold:
\begin{itemize}
\item[\(\rm i)\)] The function \(Q\ni q\mapsto\mu_q(E)\in[0,+\infty]\) is measurable for every \(E\in\Sigma\).
\item[\(\rm ii)\)] The measure \(\mu\coloneqq\int\mu_q\,\d\mathfrak q(q)\colon\Sigma\to[0,+\infty]\) is \(\sigma\)-finite, where we define
\[
\bigg(\int\mu_q\,\d\mathfrak q(q)\bigg)(E)\coloneqq\int\mu_q(E)\,\d\mathfrak q(q)\quad\text{ for every }E\in\Sigma.
\]
\end{itemize}
We say that \(q\mapsto\mu_q\) is a \textbf{disjoint} measure-valued map if there exists a family \(\{E_q\}_{q\in Q}\subseteq\Sigma\)
of pairwise disjoint sets such that \(\mu_q(\X\setminus E_q)=0\) for every \(q\in Q\) and \(\bigcup_{q\in A}E_q\in\Sigma\) for every \(A\in\mathcal Q\).
\end{definition}

If \(f\colon\X\to[0,+\infty]\) is measurable, then \(Q\ni q\mapsto\int f\,\d\mu_q\in[0,+\infty]\) is measurable and
\[
\int f\,\d\mu=\int\bigg(\int f\,\d\mu_q\bigg)\,\d\mathfrak q(q).
\]
The verification of this claim can be obtained by using the monotone convergence theorem.
\begin{example}\label{ex:trivial_disint}{\rm
Let \((Q,\mathcal Q,\mathfrak q)\) be a measure space. Let us also consider the one-point probability space
\((\{{\sf o}\},\delta_{\sf o})\). Then \(q\mapsto\delta_{(q,{\sf o})}\) is a disjoint measure-valued map
from \(Q\) to \(Q\times\{{\sf o}\}\) and
\[
\mathfrak q\otimes\delta_{\sf o}=\int\delta_{(q,{\sf o})}\,\d\mathfrak q(q).
\]
Notice also that the measure \(\mathfrak q\otimes\delta_{\sf o}\) can be canonically identified with \(\mathfrak q\).
\fr}\end{example}
\begin{remark}{\rm
Given any outer measure \(\nu\) on \(\X\), the following implication holds:
\begin{equation}
\mu_q\ll\nu\;\text{ for }\mathfrak q\text{-a.e.\ }q\in Q\quad\Longrightarrow\quad\int\mu_q\,\d\mathfrak q(q)\ll\nu.
\end{equation}
Indeed, if \(N\in\Sigma\) satisfies \(\nu(N)=0\), then trivially \(\big(\int\mu_q\,\d\mathfrak q(q)\big)(N)=\int\mu_q(N)\,\d\mathfrak q(q)=0\).
\fr}\end{remark}
\begin{definition}\label{def:L0(q)}
Let \((Q,\mathcal Q,\mathfrak q)\) be a measure space, \((\X,\Sigma)\) a measurable space, \(q\mapsto\mu_q\) a disjoint measure-valued
map from \(Q\) to \(\X\). Then we define \(L^0(\mathfrak q;L^0(\mu_\star))\coloneqq\pi_{\mathfrak q}\big(\mathcal L^0(\mathfrak q;L^0(\mu_\star))\big)\), where
\[
\mathcal L^0(\mathfrak q;L^0(\mu_\star))\coloneqq\Big\{f\in\mathscr S(L^0(\mu_\star))\;\Big|\;\exists F\in\mathcal L^0(\Sigma):\,[F]_{\mu_q}=f_q\,\text{ for }\mathfrak q\text{-a.e.\ }q\in Q\Big\}.
\]
\end{definition}

We have that, letting \(\mu\coloneqq\int\mu_q\,\d\mathfrak q(q)\) for brevity, the following identification is in force:
\[
L^0(\mathfrak q;L^0(\mu_\star))\cong L^0(\mu).
\]
The following result makes this assertion precise:
\begin{proposition}
Let \((Q,\mathcal Q,\mathfrak q)\) be a measure space, \((\X,\Sigma)\) a measurable space, \(q\mapsto\mu_q\) a disjoint measure-valued map from \(Q\) to \(\X\).
Let us define the mapping \(\bar{\sf i}_{\mathfrak q}\colon\mathcal L^0(\Sigma)\to\mathcal L^0(\mathfrak q;L^0(\mu_\star))\) as
\[
\bar{\sf i}_{\mathfrak q}(\bar F)_q\coloneqq[\bar F]_{\mu_q}\quad\text{ for every }\bar F\in\mathcal L^0(\Sigma)\text{ and }q\in Q.
\]
Moreover, letting \(\mu\coloneqq\int\mu_q\,\d\mathfrak q(q)\), we define \({\sf i}_{\mathfrak q}\colon L^0(\mu)\to L^0(\mathfrak q;L^0(\mu_\star))\) as
\[
{\sf i}_{\mathfrak q}(F)\coloneqq\pi_{\mathfrak q}\big(\bar{\sf i}_{\mathfrak q}(\bar F)\big)
\quad\text{ for every }F\in L^0(\mu)\text{ and }\bar F\in\mathcal L^0(\Sigma)\text{ such that }F=[\bar F]_\mu.
\]
Then \({\sf i}_{\mathfrak q}\) is a linear and ring isomorphism. We denote by \({\sf j}_{\mathfrak q}\colon L^0(\mathfrak q;L^0(\mu_\star))\to L^0(\mu)\) its inverse.
\end{proposition}
\begin{proof}
The map \({\sf i}_{\mathfrak q}\) is well-defined. Indeed, let $\bar{F}_1,\bar{F}_2 \in \mathcal{L}^0(\Sigma)$ be such that \([\bar{F}_1]_\mu=[\bar{F}_2]_\mu=F\). In particular, \( [\bar{F}_1]_{\mu_q}=[\bar{F}_2]_{\mu_q} \) for $\mathfrak{q}$-a.e.\ $q$. Hence, \(\pi_{\mathfrak q}\big(\bar{\sf i}_{\mathfrak q}(\bar F_1)\big)=\pi_{\mathfrak q}\big(\bar{\sf i}_{\mathfrak q}(\bar F_2)\big)\).
The fact of being a linear and ring homomorphism follows directly from similar properties
for \([\cdot]_\mu\),\([\cdot]_{\mu_q}\), \(\pi_{\mathfrak q}\).
We check surjectivity. Let \(f \in L^0(\mathfrak q;L^0(\mu_\star))\) and let \(\bar{f} \in \mathcal{L}^0(\mathfrak q;L^0(\mu_\star)) \) be such that \(f =\pi_{\mathfrak q}(\bar{f})\). Hence, there exists \(\bar{F} \in \mathcal{L}^0(\Sigma)\) such that \([\bar{F}]_{\mu_q}=\bar{f}_q\) for $\mathfrak{q}$-a.e.\ $q\in Q$.
Define $F:=[\bar{F}]_{\mu}$. Then, it can be readily checked that \({\sf i}_{\mathfrak q}(F)=f\).
\end{proof}
\begin{lemma}
Let \((Q,\mathcal Q,\mathfrak q)\) be a measure space and \((\X,\Sigma)\) a measurable space. Let \(q\mapsto\mu_q\) be a disjoint measure-valued
map from \(Q\) to \(\X\). Let \((F_n)_{n\in\N}\subseteq\mathcal L^0(\Sigma)\) and \(F\in\mathcal L^0(\Sigma)\) be given, where we denote
\(\mu\coloneqq\int\mu_q\,\d\mathfrak q(q)\). Then the following conditions are equivalent:
\begin{itemize}
\item[\(\rm i)\)] \(F(x)=\lim_n F_n(x)\) for \(\mu\)-a.e.\ \(x\in\X\).
\item[\(\rm ii)\)] For \(\mathfrak q\)-a.e.\ \(q\in Q\), it holds that \(F(x)=\lim_n F_n(x)\) for \(\mu_q\)-a.e.\ \(x\in\X\).
\end{itemize}
\end{lemma}
\begin{proof}
First of all, we point out that for any given set \(N\in\Sigma\) it holds that
\[
\mu(N)=0\quad\Longleftrightarrow\quad\mu_q(N)=0\,\text{ for }\mathfrak q\text{-a.e.\ }q\in Q.
\]
The validity of this claim immediately follows from the identity \(\mu(N)=\int\mu_q(N)\,\d\mathfrak q(q)\). In order to conclude,
choose as \(N\) the set of all \(x\in\X\) such that \((F_n(x))_{n\in\N}\) does not converge to \(F(x)\).
\end{proof}
\begin{remark}\label{rmk:wlog_mu_q_fin}{\rm
Let \((Q,\mathcal Q,\mathfrak q)\) be a measure space, \((\X,\Sigma)\) a measurable space, and \(q\mapsto\mu_q\)
a disjoint measure-valued map from \(Q\) to \(\X\). Fix a measurable function \(\phi\colon Q\to(0,1)\) such that
\(\int\phi\,\d\mathfrak q<+\infty\) and define \(\tilde{\mathfrak q}\coloneqq\phi\,\mathfrak q\). Next, fix any
\(\rho\colon\X\to(0,1)\) measurable such that \(\int\!\!\int\rho\,\d\mu_q\,\d\tilde{\mathfrak q}(q)<+\infty\). Therefore, letting \(\tilde\mu_q\coloneqq\rho\mu_q\), we have
that \(q\mapsto\tilde\mu_q\) is a disjoint measure-valued map and set \(\tilde\mu\coloneqq\int\tilde\mu_q\,\d\tilde{\mathfrak q}(q)\).
Since (almost) all the measures we marked with a tilde are finite, and we have
\[
\mathfrak q\ll\tilde{\mathfrak q}\leq\mathfrak q,\quad\mu\ll\tilde\mu\leq\mu,\quad\mu_q\ll\tilde\mu_q\leq\mu_q
\]
by construction, it will not be restrictive to assume that the measures \(\mathfrak q\), \(\mu\), \(\mu_q\) are finite.
\fr}\end{remark}
\subsection{The theory of Banach modules}
By an algebra we mean an associative, commutative algebra over the real field \(\R\) that is not necessarily unital.
In particular, an algebra is both a vector space and a ring (possibly without a multiplicative identity). For example,
given any measure space \((\X,\Sigma,\mu)\), the space \(L^0(\mu)\) is a unital algebra whose multiplicative identity
is \(\1_\X^\mu\).
\medskip

Observe that if \((\X,\Sigma,\mu)\) is a measure space and \(\mathscr A\) is a subalgebra of \(L^0(\mu)\), then the space
\[
\mathcal G(\mathscr A)\coloneqq\bigg\{\sum_{n\in\N}\1_{E_n}^\mu f_n\;\bigg|\;(E_n)_{n\in\N}\subseteq\Sigma\text{ partition of }\X,\,(f_n)_{n\in\N}\subseteq\mathscr A\bigg\}
\]
is a subalgebra of \(L^0(\mu)\). Moreover, the space \(\hat{\mathscr A}\) is a subalgebra of
\(L^0(\mu)\) as well, where we set
\[
\hat{\mathscr A}\coloneqq{\rm cl}_{L^0(\mu)}\big(\mathcal G(\mathscr A)\big).
\]
Notice that it can well happen that a subalgebra of \(L^0(\mu)\) has a multiplicative identity that differs from that of
\(L^0(\mu)\). For example, if \(E\in\Sigma\) is chosen so that \(\mm(E),\mm(\X\setminus E)>0\), then \(L^0(\mu|_E)\)
is a subalgebra of \(L^0(\mu)\) having multiplicative identity \(\1_E^\mu\neq\1_\X^\mu\).
\begin{remark}\label{rmk:about_hat_A}{\rm
Given any subalgebra \(\mathscr A\) of \(L^0(\mu)\), one can easily check that
\[
\hat{\mathscr A}\cong L^0(\mu|_{[\mathscr A]}),
\]
where \([\mathscr A]\in\Sigma\) is (\(\mu\)-a.e.) defined as the essential union of \(\{f>0\}\) as \(f\) varies in \(\mathscr A\).
\fr}\end{remark}
\begin{example}{\rm
The field \(\R\) is canonically isomorphic to the subalgebra \(\R_\mu\) of \(L^0(\mu)\) given by
\[
\R_\mu\coloneqq\big\{\lambda\1_\X^\mu\in L^0(\mu)\;\big|\;\lambda\in\R\big\}.
\]
Observe also that \(\hat\R_\mu=L^0(\mu)\) thanks to Remark \ref{rmk:about_hat_A}.
\fr}\end{example}

Next, let us introduce a notion of normed module that generalises \cite[Definition 2.6]{Gigli17}:
\begin{definition}[Normed module]\label{def:normed_module}
Let \((\X,\Sigma,\mu)\) be a measure space. Let \(\mathscr A\subseteq L^0(\mu)\) be a subalgebra.
Let \(\mathscr M\) be both a module over \(\mathscr A\) and a vector space (with the same addition operator).
Then we say that a map \(|\cdot|\colon\mathscr M\to L^0(\mu)\) is a \textbf{pointwise norm} if
\[\begin{split}
|v|\geq 0&\quad\text{ for every }v\in\mathscr M,\text{ with equality if and only if }v=0,\\
|v+w|\leq|v|+|w|&\quad\text{ for every }v,w\in\mathscr M,\\
|\lambda v|=|\lambda||v|&\quad\text{ for every }\lambda\in\R\text{ and }v\in\mathscr M,\\
|f\cdot v|=|f||v|&\quad\text{ for every }f\in\mathscr A\text{ and }v\in\mathscr M,\\
(\lambda f)\cdot v=\lambda(f\cdot v)&\quad\text{ for every }\lambda\in\R,\,f\in\mathscr A,\text{ and }v\in\mathscr M,
\end{split}\]
where all equalities and inequalities are intended in the \(\mu\)-a.e.\ sense. Whenever \(\mathscr M\)
is endowed with a pointwise norm \(|\cdot|\), we say that \(\mathscr M\) is an \textbf{\(L^0(\mu)\)-normed \(\mathscr A\)-module}.
\end{definition}

Notice that if \(\1_\X^\mu\in\mathscr A\), then \(\1_\X^\mu\cdot v=v\) for every \(v\in\mathscr M\),
whence it follows that
\[
\lambda v=(\lambda\1_\X^\mu)\cdot v\quad\text{ for every }\lambda\in\R\text{ and }v\in\mathscr M.
\]
Each pointwise norm \(|\cdot|\colon\mathscr M\to L^0(\mu)\) induces a distance \(\sfd_{\mathscr M}\) on \(\mathscr M\), in the following way:
\begin{equation}\label{eq:def_dist_Anmod}
\sfd_{\mathscr M}(v,w)\coloneqq\sfd_{L^0(\mu)}(|v-w|,0)\quad\text{ for every }v,w\in\mathscr M,
\end{equation}
\begin{definition}[Banach module]\label{def:Banach_module}
Let \((\X,\Sigma,\mu)\) be a measure space and \(\mathscr A\subseteq L^0(\mu)\) a subalgebra.
Then an \(L^0(\mu)\)-normed \(\mathscr A\)-module \(\mathscr M\) is called an \textbf{\(L^0(\mu)\)-Banach \(\mathscr A\)-module}
if \(\sfd_{\mathscr M}\) is complete.
\end{definition}

Notice that the \(L^0(\mu)\)-Banach \(\mathscr A\)-modules with \(\mathscr A=L^0(\mu)\) are exactly the
\(L^0\)-normed modules in the sense of \cite[Definition 2.6]{Gigli17} (indeed, in \cite{Gigli17}
normed modules are assumed to be complete).
\begin{definition}[Generators]
\label{def:generators}
Let \((\X,\Sigma,\mu)\) be a measure space and \(\mathscr M\) an \(L^0(\mu)\)-Banach \(L^0(\mu)\)-module.
Then a vector subspace \(V\subseteq\mathscr M\) \textbf{generates} \(\mathscr M\) if \(\mathcal G(V)\) is
dense in \(\mathscr M\), where we set
\begin{equation}\label{eq:def_G(V)}
\mathcal G(V)\coloneqq\bigg\{\sum_{n\in\N}\1_{E_n}^\mu\cdot v_n\;\bigg|\;(E_n)_{n\in\N}\subseteq\Sigma\text{ partition of }\X,
\,(v_n)_{n\in\N}\subseteq V\bigg\}.
\end{equation}
Moreover, we say that \(\mathscr M\) is \textbf{countably generated} if a countable vector subspace generates \(\mathscr M\). 
\end{definition}
\begin{example}\label{ex:Ban_as_mod}{\rm
Let \((\{{\sf o}\},\delta_{\sf o})\) be the one-point probability space. Then we have that
\[
\big\{L^0(\delta_{\sf o})\text{-Banach }L^0(\delta_{\sf o})\text{-modules}\big\}=\{\text{Banach spaces}\},
\]
in the sense which we are going to explain. First, observe that \(L^0(\delta_{\sf o})=\R_{\delta_{\sf o}}\)
can be canonically identified with \(\R\). Given an \(L^0(\delta_{\sf o})\)-Banach \(L^0(\delta_{\sf o})\)-module
\(\mathscr M\), we thus have that the pointwise norm
\(\|\cdot\|\coloneqq|\cdot|\colon\mathscr M\to L^0(\delta_{\sf o})\cong\R\) can be regarded
as a norm. The resulting normed space \(\mathscr M\) is in fact Banach, as follows from the completeness
of the distance \(\sfd_{\mathscr M}(v,w)=\|v-w\|\wedge 1\).
\fr}\end{example}
\begin{proposition}[Completion of a normed module]
Let \((\X,\Sigma,\mu)\) be a measure space and \(\mathscr A\) a subalgebra of \(L^0(\mu)\) such that \(\hat{\mathscr A}=L^0(\mu)\).
Let \(\mathscr M\) be an \(L^0(\mu)\)-normed \(\mathscr A\)-module. Then there exists a unique couple
\((\hat{\mathscr M},\iota)\) having the following properties:
\begin{itemize}
\item[\(\rm i)\)] \(\hat{\mathscr M}\) is an \(L^0(\mu)\)-Banach \(L^0(\mu)\)-module,
\item[\(\rm ii)\)] \(\iota\colon\mathscr M\to\hat{\mathscr M}\) is a linear map with generating image
such that \(|\iota(v)|=|v|\) for every \(v\in\mathscr M\).
\end{itemize}
The couple \((\hat{\mathscr M},\iota)\) is unique up to a unique isomorphism: given any couple
\((\mathscr N,\tilde\iota)\) with the same properties, there is a unique isomorphism \(\Phi\colon\hat{\mathscr M}\to\mathscr N\)
of \(L^0(\mu)\)-Banach \(L^0(\mu)\)-modules such that
\[\begin{tikzcd}
\mathscr M \arrow[r,"\iota"] \arrow[rd,swap,"\tilde\iota"] & \hat{\mathscr M} \arrow[d,"\Phi"] \\
& \mathscr N
\end{tikzcd}\]
is a commutative diagram. We say that \(\hat{\mathscr M}\) is the \textbf{completion} of \(\mathscr M\).
\end{proposition}
\begin{proof}
Recall Remark \ref{rmk:about_hat_A}. The triple
\(\big(L^0(\mu),L^0(\mu),L^0(\mu)\big)\) is a metric \(f\)-structure in the sense of \cite[Definition 2.24]{LP23}; cf.\ with \cite[Section 4.2.2]{LP23}.
Defining \(\mathscr V\coloneqq\mathscr M\) and the map \(\psi\colon\mathscr V\to L^0(\mu)^+\) as \(\psi(v)\coloneqq|v|\), we are in a position to apply
\cite[Theorem 3.19]{LP23} to \((\mathscr V,\psi)\). Then
\((\hat{\mathscr M},\iota)\coloneqq(\mathscr M_{\langle\psi\rangle},T_{\langle\psi\rangle})\) verifies the statement.
\end{proof}
\subsubsection{Banach modules with respect to a submodular outer measure}
Let \((\X,\sfd)\) be a metric space. We denote by \(\mathscr B(\X)\) the Borel \(\sigma\)-algebra of \(\X\).
Fix an outer measure \(\nu\) on \(\X\). Given a Borel function \(f\colon\X\to[0,+\infty]\) and a set
\(E\in\mathscr B(\X)\), one can define the integral of \(f\) on \(E\) with respect to \(\nu\) via Cavalieri's formula, i.e.
\begin{equation}\label{eq:def_int_outer}
\int_E f\,\d\nu\coloneqq\int_0^{+\infty}\nu(\{\1_E f\geq t\})\,\d t.
\end{equation}
Moreover, we say that:
\begin{itemize}
\item \(\nu\) is \textbf{boundedly-finite} if \(\nu(B)<+\infty\) for every \(B\in\mathscr B(\X)\) bounded.
\item \(\nu\) is \textbf{submodular} if \(\nu(E\cup F)+\nu(E\cap F)\leq\nu(E)+\nu(F)\) for every \(E,F\in\mathscr B(\X)\).
\end{itemize}
It holds that the integral defined in \eqref{eq:def_int_outer} is subadditive, which means that
\[
\int_\X (f+g)\,\d\nu\leq\int_\X f\,\,\d\nu+\int_\X g\,\d\nu\quad\text{ for every }f,g\colon\X\to[0,+\infty]\text{ Borel,}
\]
if and only if \(\nu\) is submodular. See \cite{debin2019quasicontinuous} after \cite{denneberg2010non}.
\begin{example}\label{ex:submod_outer}{\rm
We are concerned with two types of boundedly-finite, submodular outer measures:
\begin{itemize}
\item[\(\rm i)\)] The outer measure \(\bar\mu\) induced via the Carath\'{e}odory construction, i.e.
\[
\bar\mu(S)\coloneqq\inf\big\{\mu(E)\;\big|\;E\in\mathscr B(\X),\,S\subseteq E\big\}\quad\text{ for every }S\subseteq\X,
\]
by a boundedly-finite Borel measure \(\mu\geq 0\) on a metric space \((\X,\sfd)\).
\item[\(\rm ii)\)] The Sobolev capacity \(\rm Cap\) on a metric measure space \((\X,\sfd,\mm)\).
\end{itemize}
We will introduce the Sobolev capacity in Section \ref{s:BV}.
\fr}\end{example}
Fix a metric space \((\X,\sfd)\) and a boundedly-finite, submodular outer measure \(\nu\) on \(\X\). Then we denote by \(L^0(\nu)\) the space
of all Borel functions from \(\X\) to \(\R\), quotiented up to \(\nu\)-a.e.\ equality. In order to define a distance on \(L^0(\nu)\), fix an
increasing sequence \((U_n)_{n\in\N}\) of bounded open subsets of \(\X\) with the following property: for any bounded set \(B\subseteq\X\),
there exists \(n\in\N\) such that \(B\subseteq U_n\). In particular, we have that \(\X=\bigcup_{n\in\N}U_n\). Then we define
\[
\sfd_{L^0(\nu)}(f,g)\coloneqq\sum_{n\in\N}\frac{1}{2^n(\nu(U_n)\vee 1)}\int_{U_n}|f-g|\wedge 1\,\d\nu\quad\text{ for every }f,g\in L^0(\nu).
\]
The submodularity of the integration with respect to \(\nu\) ensures that \(\sfd_{L^0(\nu)}\) is a distance. Even though the distance \(\sfd_{L^0(\nu)}\)
depends on the chosen sequence \((U_n)_n\), its induced topology does not and makes \(L^0(\nu)\) into a topological vector space. Given a Borel measure
\(\mu\) on \(\X\) such that \(\mu\ll\nu\), we denote by \(\pi_\mu\colon L^0(\nu)\to L^0(\mu)\) the canonical projection map and by \([f]_\mu\in L^0(\mu)\)
the equivalence class of a function \(f\in L^0(\nu)\). Moreover, in the framework of Example \ref{ex:submod_outer} i), we have that \(L^0(\bar\mu)=L^0(\mu)\) as topological vector spaces.
\medskip

The following definition is adapted from \cite[Definition 3.1]{debin2019quasicontinuous}:
\begin{definition}
Let \((\X,\sfd)\) be a metric space and \(\nu\) a boundedly-finite, submodular outer measure on \(\X\).
Let \(\mathscr M\) be a module over \(L^0(\nu)\). Then \(|\cdot|\colon\mathscr M\to L^0(\nu)\) is a \textbf{pointwise norm} if
\[\begin{split}
|v|\geq 0&\quad\text{ for every }v\in\mathscr M,\text{ with equality if and only if }v=0,\\
|v+w|\leq|v|+|w|&\quad\text{ for every }v,w\in\mathscr M,\\
|f\cdot v|=|f||v|&\quad\text{ for every }f\in L^0(\nu)\text{ and }v\in\mathscr M,
\end{split}\]
where all equalities and inequalities are intended in the \(\nu\)-a.e.\ sense. Whenever \(\mathscr M\)
is endowed with a pointwise norm \(|\cdot|\), we say that \(\mathscr M\) is an \textbf{\(L^0(\nu)\)-normed \(L^0(\nu)\)-module}. When the distance
\[
\sfd_{\mathscr M}(v,w)\coloneqq\sfd_{L^0(\nu)}(|v-w|,0)\quad\text{ for every }v,w\in\mathscr M
\]
is a complete distance on \(\mathscr M\), we say that \(\mathscr M\) is an \textbf{\(L^0(\nu)\)-Banach \(L^0(\nu)\)-module}.
\end{definition}

In view of Example \ref{ex:submod_outer} i), the above notion of a normed/Banach module generalises Definitions \ref{def:normed_module} and \ref{def:Banach_module} with \(\mathscr A=L^0(\mu)\).
Generalising Definition \ref{def:generators}, we say that a vector subspace \(V\subseteq\mathscr M\)
\textbf{generates} \(\mathscr M\) if \(\mathcal G(V)\) is dense in \(\mathscr M\), where \(\mathcal G(V)\) is defined
as in \eqref{eq:def_G(V)} (with \(\nu\) in place of \(\mu\)).
\medskip

In the case where \(\mu\ll\nu\) and \(\mathscr M\) is a given \(L^0(\nu)\)-Banach \(L^0(\nu)\)-module, we can `quotient \(\mathscr M\) up to \(\mu\)-a.e.\ equality', i.e.\ there is a natural
\(L^0(\mu)\)-Banach \(L^0(\mu)\)-module structure on the quotient
\begin{equation}\label{eq:def_quot_M}
\mathscr M_\mu\coloneqq\mathscr M/\sim_\mu,\quad\text{ where }v\sim_\mu w\text{ if and only if }|v-w|=0\text{ holds }\mu\text{-a.e.\ on }\X.
\end{equation}
In analogy with (and, in fact, generalising) the notation for the spaces of functions, we denote by \(\pi_\mu\colon\mathscr M\to\mathscr M_\mu\)
the canonical projection map and by \([v]_\mu\in\mathscr M_\mu\) the equivalence class of \(v\in\mathscr M\).
\subsection{Calculus on metric measure spaces}
By a \textbf{metric measure space} \((\X,\sfd,\mm)\) we mean a complete, separable metric space \((\X,\sfd)\) equipped with a boundedly finite Borel
measure \(\mm\geq 0\). We denote by \({\rm LIP}(\X)\) (resp.\ \({\rm LIP}_{loc}(\X)\), resp.\ \({\rm LIP}_{bs}(\X)\)) the space of all Lipschitz (resp.\ locally Lipschitz,
resp.\ boundedly-supported Lipschitz) functions from \(\X\) to \(\R\). The \textbf{asymptotic slope} \({\rm lip}_a(f)\colon\X\to[0,+\infty)\) of \(f\in{\rm LIP}_{loc}(\X)\) is defined as
\[
{\rm lip}_a(f)(x)\coloneqq\lims_{y,z\to x}\frac{|f(y)-f(z)|}{\sfd(y,z)}\quad\text{ for every accumulation point }x\in\X
\]
and \({\rm lip}_a(f)(x)\coloneqq 0\) for every isolated point \(x\in\X\). The \textbf{Cheeger energy} of \((\X,\sfd,\mm)\), which was introduced in \cite{AmbrosioGigliSavare11-3} after \cite{Cheeger00},
is the functional \({\rm Ch}\colon L^2(\mm)\to[0,\infty]\) given by
\[
{\rm Ch}(f)\coloneqq\frac{1}{2}\inf\bigg\{\limi_{n\to\infty}\int{\rm lip}_a^2(f_n)\,\d\mm\;\bigg|\;(f_n)_{n\in\N}\subseteq{\rm LIP}_{bs}(\X),\,\pi_\mm(f_n)\to f\text{ in }L^2(\mm)\bigg\}
\]
for every \(f\in L^2(\mm)\). The \textbf{Sobolev space} \(H^{1,2}(\X)\) is then defined as
\[
H^{1,2}(\X)\coloneqq\big\{f\in L^2(\mm)\;\big|\;{\rm Ch}(f)<+\infty\big\}.
\]
Given any \(f\in H^{1,2}(\X)\), there exists a unique non-negative function \(|Df|\in L^2(\mm)\), which is called the \textbf{minimal relaxed slope} of \(f\), providing the integral representation
\[
{\rm Ch}(f)=\frac{1}{2}\int|Df|^2\,\d\mm.
\]
It holds that \(\pi_\mm({\rm LIP}_{bs}(\X))\subseteq H^{1,2}(\X)\) and \(|Df|\leq{\rm lip}_a(f)\) in the  \(\mm\)-a.e.\ sense for all \(f\in{\rm LIP}_{bs}(\X)\).
The Sobolev space is a Banach space if endowed with the norm
\[
\|f\|_{H^{1,2}(\X)}\coloneqq\big(\|f\|_{L^2(\mm)}^2+\||Df|\|_{L^2(\mm)}^2\big)^{1/2}\quad\text{ for every }f\in H^{1,2}(\X).
\]
However, \(\big(H^{1,2}(\X),\|\cdot\|_{H^{1,2}(\X)}\big)\) needs not be a Hilbert space. Following \cite{Gigli12}, we say that
\[
(\X,\sfd,\mm)\text{ is \textbf{infinitesimally Hilbertian} if }H^{1,2}(\X)\text{ is a Hilbert space.}
\]
In the infinitesimally Hilbertian setting, the \textbf{carr\'{e} du champ} operator
\[
H^{1,2}(\X)\times H^{1,2}(\X)\ni(f,g)\mapsto\langle\nabla f,\nabla g\rangle\coloneqq\frac{|D(f+g)|^2-|Df|^2-|Dg|^2}{2}\in L^1(\mm)
\]
is an \(L^\infty(\mm)\)-bilinear operator. We say that a function \(f\in H^{1,2}(\X)\) has a \textbf{Laplacian} if there exists a (necessarily unique) function
\(\Delta f\in L^2(\mm)\) such that
\[
\int\langle\nabla f,\nabla g\rangle\,\d\mm=-\int g\,\Delta f\,\d\mm\quad\text{ for every }g\in H^{1,2}(\X).
\]
We denote by \(D(\Delta)\) the space of all those functions \(f\in H^{1,2}(\X)\) having a Laplacian.
\medskip

The following definition is essentially taken from \cite{Gigli14,Gigli17}:
\begin{definition}[Tangent module]
Let \((\X,\sfd,\mm)\) be an infinitesimally Hilbertian metric measure space. Then there exists a unique (up to a unique isomorphism) couple \((L^0(T\X),\nabla)\), where:
\begin{itemize}
\item[\(\rm i)\)] \(L^0(T\X)\) is an \(L^0(\mm)\)-Banach \(L^0(\mm)\)-module.
\item[\(\rm ii)\)] \(\nabla\colon H^{1,2}(\X)\to L^0(T\X)\) is a linear operator such that
\[\begin{split}
|\nabla f|=|Df|&\quad\text{ for every }f\in H^{1,2}(\X),\\
\{\nabla f\,:\,f\in H^{1,2}(\X)\}&\quad\text{ generates }L^0(T\X).
\end{split}\]
\end{itemize}
The space \(L^0(T\X)\) is called the \textbf{tangent module} of \((\X,\sfd,\mm)\) and \(\nabla\) the \textbf{gradient operator}.
\end{definition}

Thanks to the locality properties of the minimal relaxed slopes, the gradient operator can be extended to all Lipschitz functions on \(\X\).
\subsubsection{Functions of bounded variation on PI spaces}\label{s:BV}
We recall the notion of function of bounded variation on a metric measure space. The following definition was introduced in \cite{Ambrosio-DiMarino14} after \cite{MIRANDA2003}.
\begin{definition}[BV space]
Let \((\X,\sfd,\mm)\) be a metric measure space and \(f\in L^1_{loc}(\X)\). Given any open set \(U\subseteq\X\), we define the \textbf{total variation} of \(f\) on \(U\) as
\begin{equation}\label{eq:def_BV}
|Df|(U)\coloneqq\inf\bigg\{\limi_{n\to\infty}\int_U{\rm lip}_a(f_n)\,\d\mm\;\bigg|\;(f_n)_{n\in\N}\subseteq{\rm LIP}_{loc}(U)\cap L^1(U),\,f_n\to f\text{ in }L^1_{loc}(U)\bigg\}.
\end{equation}
Then we say that \(f\) is \textbf{of bounded variation} if \(|Df|(\X)<+\infty\). We define the space \(BV(\X)\) as
\[
BV(\X)\coloneqq\big\{f\in L^1(\mm)\;\big|\;|Df|(\X)<+\infty\big\}.
\]
Moreover, we say that \(E\subseteq\X\) Borel is a \textbf{set of finite perimeter} if \(P(E)\coloneqq|D\1_E|(\X)<+\infty\).
\end{definition}

Given any function \(f\in L^1_{loc}(\X)\) of bounded variation, there exists a unique finite Borel measure on \(\X\), still denoted by \(|Df|\), which extends the set-function
on open sets we defined in \eqref{eq:def_BV}. In the case where \(f=\1_E\) for some set of finite perimeter \(E\subseteq\X\), we write \(P(E,\cdot)\coloneqq|D\1_E|\) and we call
\(P(E,\cdot)\) the \textbf{perimeter measure} of \(E\). The following result is taken from \cite{MIRANDA2003}:
\begin{theorem}[Coarea formula]\label{thm:coarea}
Let \((\X,\sfd,\mm)\) be a metric measure space and let \(f\in BV(\X)\). Then \(t\mapsto P(\{f>t\},\cdot)\) is a measurable measure-valued map from \(\R\) to \(\X\). Moreover, it holds that
\[
|Df|=\int_\R P(\{f>t\},\cdot)\,\d t.
\]
\end{theorem}

Following \cite[Definition 2.6]{debin2019quasicontinuous} (which is a variant of \cite[formula (7.2.1)]{HKST15}), we define the \textbf{Sobolev capacity} on the space \((\X,\sfd,\mm)\) as
\[
{\rm Cap}(E)\coloneqq\inf_f\|f\|_{H^{1,2}(\X)}^2\quad\text{ for every set }E\subseteq\X,
\]
where the infimum is taken over all functions \(f\in H^{1,2}(\X)\) satisfying \(f\geq 1\) \(\mm\)-a.e.\ on some open neighbourhood of \(E\).
The Sobolev capacity \(\rm Cap\) is a boundedly-finite, submodular outer measure on \(\X\) such that \(\mm\ll{\rm Cap}\). Moreover, it was proved in \cite{BGBV} after \cite{bru2019rectifiability} that
\begin{equation}\label{eq:ac_Cap}
|Df|\ll{\rm Cap}\quad\text{ for every }f\in L^1_{loc}(\X)\text{ of bounded variation,}
\end{equation}
thus in particular \(P(E,\cdot)\ll{\rm Cap}\) for every set of finite perimeter \(E\subseteq\X\).
\medskip

Even though the basic theory of BV functions is meaningful on arbitrary metric measure spaces, a much more refined analysis is available in the setting of PI spaces,
which we are going to introduce. We refer e.g.\ to \cite{Bjorn-Bjorn11,HKST15} for a thorough account of this topic.
\begin{definition}[PI space]
Let \((\X,\sfd,\mm)\) be a metric measure space. Then:
\begin{itemize}
\item[\(\rm i)\)] We say that \((\X,\sfd,\mm)\) is \textbf{uniformly locally doubling} if there exists a non-decreasing function \(C_D\colon(0,+\infty)\to(0,+\infty)\) such that
\[
\mm(B_{2r}(x))\leq C_D(R)\mm(B_r(x))\quad\text{ for every }0<r<R\text{ and }x\in\X.
\]
\item[\(\rm ii)\)] We say that \((\X,\sfd,\mm)\) supports a \textbf{weak local \((1,1)\)-local Poincar\'{e} inequality} if there exist a constant \(\lambda\geq 1\) and a
non-decreasing function \(C_P\colon(0,+\infty)\to(0,+\infty)\) such that for any function \(f\in{\rm LIP}(\X)\) it holds that
\[
\fint_{B_r(x)}\bigg|f-\fint_{B_r(x)}f\,\d\mm\bigg|\,\d\mm\leq C_P(R)\,r\fint_{B_{\lambda r}(x)}{\rm lip}_a(f)\,\d\mm
\]
for every \(0<r<R\) and \(x\in\X\).
\item[\(\rm iii)\)] We say that \((\X,\sfd,\mm)\) is a \textbf{PI space} if it is uniformly locally doubling and it supports a weak local \((1,1)\)-Poincar\'{e} inequality.
\end{itemize}
\end{definition}

Let \((\X,\sfd,\mm)\) be a PI space. The \textbf{upper density} of a Borel set \(E\subseteq\X\) at a point \(x\in\X\) is
\[
\Theta^*(E,x)\coloneqq\lims_{r\searrow 0}\frac{\mm(E\cap B_r(x))}{\mm(B_r(x))}\in[0,1].
\]
The \textbf{essential boundary} of \(E\) is defined as the Borel set
\[
\partial^*E\coloneqq\big\{x\in\X\;\big|\;\Theta^*(E,x),\Theta^*(\X\setminus E,x)>0\big\}.
\]
Following \cite{Amb01}, we define the \textbf{codimension-one Hausdorff measure} of a set \(E\subseteq\X\) as
\[
\mathcal H^h(E)\coloneqq\lim_{\delta\searrow 0}\mathcal H^h_\delta(E),
\]
where for any \(\delta>0\) we define
\[
\mathcal H^h_\delta(E)\coloneqq\inf\bigg\{\sum_{n\in\N}\frac{\mm(B_{r_n}(x_n))}{2r_n}\;\bigg|\;(x_n)_n\subseteq\X,\,(r_n)_n\subseteq(0,\delta),\,E\subseteq\bigcup_{n\in\N}B_{r_n}(x_n)\bigg\}.
\]
Both \(\mathcal H^h_\delta\) and \(\mathcal H^h\) are Borel regular outer measures on \(\X\). If \(E\subseteq\X\) is a set of finite perimeter, then it was proved
in \cite{Amb01} (and \cite{ambmirpal04}) that the perimeter measure of \(E\) can be written as
\[
P(E,\cdot)=\theta_E\mathcal H^h|_{\partial^*E},
\]
for some Borel density function \(\theta_E\colon\X\to[\gamma,C]\), where the constants \(C\geq\gamma>0\) depend exclusively
on \(C_D(\cdot)\), \(C_P(\cdot)\), and \(\lambda\). In particular, the perimeter measure \(P(E,\cdot)\) is concentrated on \(\partial^*E\).
\begin{definition}[Precise representative]
Let \((\X,\sfd,\mm)\) be a PI space. Let \(f\in BV(\X)\) be a given function. Then the \emph{approximate lower limit}
\(f^\wedge\colon\X\to[-\infty,+\infty]\) and the \emph{approximate upper limit} \(f^\vee\colon\X\to[-\infty,+\infty]\)
of \(f\) are defined as
\[\begin{split}
f^\wedge(x)&\coloneqq\sup\big\{t\in\R\;\big|\;\Theta^*(\{f<t\},x)=0\big\},\\
f^\vee(x)&\coloneqq\inf\big\{t\in\R\;\big|\;\Theta^*(\{f>t\},x)=0\big\}
\end{split}\]
for every \(x\in\X\), respectively. We also define the Borel set \(\X_f\subseteq\X\) as
\[
\X_f\coloneqq\{f^\wedge>-\infty\}\cap\{f^\vee<+\infty\}.
\]
Finally, the \emph{precise representative} \(\bar f\colon\X_f\to\R\) of \(f\) is the Borel function given by
\[
\bar f(x)\coloneqq\frac{f^\wedge(x)+f^\vee(x)}{2}\quad\text{ for every }x\in\X_f.
\]
\end{definition}
\begin{remark}{\rm
Since \(\{f>t\}\subseteq\{f\geq t\}\subseteq\{f>s\}\) for every \(s,t\in\R\) with \(s<t\), one has that
\[\begin{split}
f^\wedge(x)&=\sup\big\{t\in\R\;\big|\;\Theta^*(\{f\leq t\},x)=0\big\},\\
f^\vee(x)&=\inf\big\{t\in\R\;\big|\;\Theta^*(\{f\geq t\},x)=0\big\}
\end{split}\]
for every \(x\in\X\).
\fr}\end{remark}

It holds that \(\mm(\X\setminus\X_f)=|Df|(\X\setminus\X_f)=0\), so that \(|Df||_{\X_f}=|Df|\),
and that \(\bar f\) is an \(\mm\)-a.e.\ representative of \(f\).
The \emph{jump set} \(J_f\subseteq\X_f\) of \(f\) is defined as
\[
J_f\coloneqq\big\{x\in\X\;\big|\;f^\wedge(x)<f^\vee(x)\big\}.
\]
It holds that \(J_f\) is contained in a countable union of essential boundaries of sets of finite perimeter, thus
in particular \(\mm(J_f)=0\). More precisely, for suitably chosen \((t_n)_n,(s_n)_n\subseteq\R\) with \(t_n\neq s_n\),
\[
J_f=\bigcup_{n\in\N}\partial^*\{\bar f>t_n\}\cap\partial^*\{\bar f>s_n\}.
\]
Furthermore, there exists a Borel function \(\theta_f\colon\X\to[\gamma,C]\) such that
\[
|Df||_{J_f}=(f^\vee-f^\wedge)\theta_f\mathcal H^h|_{J_f}.
\]
The following definition has been proposed in \cite[Definition 6.1]{AMP04}.
We say that the metric measure space $(\X,\sfd,\mm)$ is \emph{isotropic} provided that, given two sets of finite perimeter $E$ and $F$, it holds that
\begin{equation}
    \theta_E =\theta_F \quad\mathcal H^h\text{-a.e. in }\partial^* E \cap \partial^* F.
\end{equation}
In the case where \((\X,\sfd,\mm)\) is isotropic, we have that for any set \(G\subseteq\X\) of finite perimeter it holds
\begin{equation}\label{eq:jump_isotr}
|Df||_{J_f\cap\partial^*G}=(f^\vee-f^\wedge)P(G,\cdot)|_{J_f}.
\end{equation}
We refer to \cite{ambrosio2004topics,kinkorshatuo} for the proofs of the above claims.
\subsection{Calculus on \texorpdfstring{\(\sf RCD\)}{RCD} spaces}
We assume the reader is familiar with the language of \({\sf RCD}(K,N)\) spaces. Let us only recall that an \({\sf RCD}(K,N)\) space \((\X,\sfd,\mm)\), with \(K\in\R\) and \(N\in[1,\infty)\),
is an infinitesimally Hilbertian metric measure space whose synthetic Ricci curvature (resp.\ synthetic dimension) is bounded from below by \(K\) (resp.\ from above by \(N\)),
in the sense of Lott--Sturm--Villani \cite{Lott-Villani09,Sturm06II}. Each \({\sf RCD}(K,N)\) space is a PI space (see \cite{Sturm06II} for the doubling
condition and \cite{Rajala12} for the Poincar\'{e} inequality).
Moreover, ${\sf RCD}(K,N)$ spaces are isotropic. This has been observed in \cite[Example 1.29]{BonPasRaj2020}.
For more of this topic, see the survey \cite{AmbICM} and the references therein.
\medskip

Following \cite{Gigli14} (after \cite{Savare13}), we consider the distinguished algebra of \textbf{test functions} on \(\X\):
\[
{\rm Test}(\X)\coloneqq\big\{f\in{\rm LIP}(\X)\cap D(\Delta)\cap L^\infty(\mm)\;\big|\;|Df|\in L^\infty(\mm),\,\Delta f\in H^{1,2}(\X)\big\}.
\]
The space \({\rm Test}(\X)\) is dense in \(H^{1,2}(\X)\). Moreover, the following property holds:
\begin{lemma}
Let \((\X,\sfd,\mm)\) be an \({\sf RCD}(K,N)\) space. Then \({\rm Test}(\X)\) generates \(L^0({\rm Cap})\).
\end{lemma}
\begin{proof}
In fact, we will prove a stronger statement, namely that if \(f\colon\X\to\R\) is a Borel function and \(\varepsilon>0\) is given, then there exists a Borel partition
\((E_n)_n\) of \(\X\) and \((f_n)_n\subseteq{\rm Test}(\X)\) such that
\begin{equation}\label{eq:test_gen}
|f(x)-f_n(x)|\leq\varepsilon\quad\text{ for every }n\in\N\text{ and }x\in E_n.
\end{equation}
First, we can find a Borel partition \((E_n)_n\) of \(\X\) into bounded sets and a sequence \((\lambda_n)_n\subseteq\R\) such that \(|f-\lambda_n|\leq\varepsilon\) on \(E_n\).
Using the results of \cite{AmbrosioMondinoSavare13-2}, for any \(n\in\N\) we can find a test cut-off function \(f_n\in{\rm Test}(\X)\) satisfying \(f_n=\lambda_n\) on an open ball
containing \(E_n\). Property \eqref{eq:test_gen} follows.
\end{proof}

Given any \(f,g\in{\rm Test}(\X)\), we have that
\[
\langle\nabla f,\nabla g\rangle\in H^{1,2}(\X).
\]
In particular, we deduce that for any \(f\in{\rm Test}(\X)\) the function \(|Df|\) admits a \textbf{quasi-continuous representative} \({\rm QCR}(|Df|)\in L^0({\rm Cap})\), cf.\ with \cite{debin2019quasicontinuous}.
Therefore, the following result, which is taken from \cite[Theorem 3.6]{debin2019quasicontinuous}, is meaningful:
\begin{theorem}[Capacitary tangent module]
Let \((\X,\sfd,\mm)\) be an \({\sf RCD}(K,N)\) space. Then there exists a unique (up to a unique isomorphism) couple \((L^0_{\rm Cap}(T\X),\bar\nabla)\), where:
\begin{itemize}
\item[\(\rm i)\)] \(L^0_{\rm Cap}(T\X)\) is an \(L^0({\rm Cap})\)-Banach \(L^0({\rm Cap})\)-module.
\item[\(\rm ii)\)] \(\bar\nabla\colon{\rm Test}(\X)\to L^0_{\rm Cap}(T\X)\) is a linear operator such that
\[\begin{split}
|\bar\nabla f|={\rm QCR}(|Df|)&\quad\text{ for every }f\in{\rm Test}(\X),\\
\{\bar\nabla f\,:\,f\in{\rm Test}(\X)\}&\quad\text{ generates }L^0_{\rm Cap}(T\X).
\end{split}\]
\end{itemize}
The space \(L^0_{\rm Cap}(T\X)\) is called the \textbf{capacitary tangent module} of \((\X,\sfd,\mm)\).
\end{theorem}

The space \({\rm Test}(T\X)\subseteq L^0_{\rm Cap}(T\X)\) of \textbf{test vector fields} is then defined as
\[
{\rm Test}(T\X)\coloneqq\bigg\{\sum_{i=1}^n g_i\bar\nabla f_i\;\bigg|\;n\in\N,\,(f_i)_{i=1}^n,(g_i)_{i=1}^n\subseteq{\rm Test}(\X)\bigg\}.
\]
The space \({\rm Test}(T\X)\) generates \(L^0_{\rm Cap}(T\X)\). Given a Borel measure \(\mu\) on \(\X\) with \(\mu\ll{\rm Cap}\), we set
\[
L^0_\mu(T\X)\coloneqq L^0_{\rm Cap}(T\X)_\mu,
\]
where the right-hand side is defined as in \eqref{eq:def_quot_M}. One can readily check that
\[
L^0_\mm(T\X)\cong L^0(T\X),\quad\text{ with }\pi_\mm\circ\bar\nabla=\nabla|_{{\rm Test}(\X)}.
\]
Thanks to \eqref{eq:ac_Cap}, we can consider \(L^0_{|Df|}(T\X)\) whenever \(f\in L^1_{loc}(\X)\) is of bounded variation.
\section{Integration of Banach modules}\label{sec:banach_module_bundles}
In this section, we introduce a notion of Banach module bundle consistent with a disjoint measure-valued map, which generalises the notion
of strong Banach bundle studied in \cite{LGP22}.
\medskip

First, we introduce a useful shorthand notation. Let \((Q,\mathcal Q,\mathfrak q)\) be a measure space, \((\X,\Sigma)\) a measurable space, and
\(q\mapsto\mu_q\) a disjoint measure-valued map from \(Q\) to \(\X\). Given any collection \(\mathscr M_\star=\{\mathscr M_q\}_{q\in Q}\), where \(\mathscr M_q\)
is an \(L^0(\mu_q)\)-Banach \(L^0(\mu_q)\)-module, and for any element \(v\) of \(\mathscr S(\mathscr M_\star)\) or of \(\mathscr S_{\mathfrak q}(\mathscr M_\star)\), we denote by
\[
|v_\star|\in\mathscr S_{\mathfrak q}(L^0(\mu_\star))
\]
the (\(\mathfrak q\)-a.e.\ equivalence class of the) map sending \(q\in Q\) to \(|v_q|\in L^0(\mu_q)\).
\begin{definition}[Banach module bundle]
\label{def:Banach_module_bundle}
Let \((Q,\mathcal Q,\mathfrak q)\) be a measure space, \((\X,\Sigma)\) a measurable space, and \(q\mapsto\mu_q\) a disjoint measure-valued map
from \(Q\) to \(\X\). Let \((\mathscr M_\star,\Theta,\Phi)\) be such that:
\begin{itemize}
\item[\(\rm a)\)] \(\mathscr M_\star=\{\mathscr M_q\}_{q\in Q}\) and each space \(\mathscr M_q\)
is an \(L^0(\mu_q)\)-Banach \(L^0(\mu_q)\)-module.
\item[\(\rm b)\)] \(\Phi\) is a subalgebra of \(L^0(\mathfrak q;L^0(\mu_\star))\) such that the subalgebra \(\mathcal G_{\mathfrak q}(\Phi)\) is dense in \(L^0(\mu)\), where
\[
\mathcal G_{\mathfrak q}(\Phi)\coloneqq\bigg\{\sum_{n\in\N}{\sf j}_{\mathfrak q}(\1_{S_n}^{\mathfrak q}f_n)\;\bigg|\;(S_n)_{n\in\N}\subseteq\mathcal Q
\text{ partition of }Q,\,(f_n)_{n\in\N}\subseteq\Phi\bigg\}\subseteq L^0(\mu).
\]
In particular, it holds that \(\widehat{{\sf j}_{\mathfrak q}(\Phi)}=L^0(\mu)\).
\item[\(\rm c)\)] \(\Theta\) is a vector subspace of \(\mathscr S_{\mathfrak q}(\mathscr M_\star)\) that
is also a module over \(\Phi\) and verifies
\[
|v_\star|\in L^0(\mathfrak q;L^0(\mu_\star))\quad\text{ for every }v\in\Theta.
\]
\end{itemize}
Then we say that \((\mathscr M_\star,\Theta,\Phi)\) is a \textbf{Banach module bundle} consistent with \(q\mapsto\mu_q\).
\end{definition}

\begin{remark}{\rm
Let us clarify why the elements of $\mathcal{G}_\q(\Phi)$ are well-posed. Fix a partition $(S_n)_n\subseteq \mathcal{Q}$ of $Q$ and a sequence $(f_n)_n\subseteq\Phi$.
Let $F_n\in L^0(\mu)$, with representative $\bar F_n\in\mathcal{L}^0(\Sigma)$ be such that ${\rm i}_{\mathfrak q}(F_n)=f_n$. Then \({\sf i}_\q(\1_{U_n}^\mu F_n)=\1_{S_n}^\q f_n\in L^0(\q;L^0(\mu_\star))\),
where \(U_n\coloneqq\bigcup_{q\in S_n}E_q\in\Sigma\). Furthermore, \((U_n)_n\) are pairwise disjoint, \(\sum_{n\in\N}\1_{U_n}\bar F_n\in\mathcal L^0(\Sigma)\), and
\(\sum_{n\in \N}{\sf j}_q(\1_{S_n}^\q f_n)=\big[\sum_{n\in \N}\1_{U_n}\bar F_n\big]_\mu\).
\fr}\end{remark}

One can readily check that the mapping \(|\cdot|\colon\Theta\to L^0(\mu)\), which we define as
\[
|v|\coloneqq{\sf j}_{\mathfrak q}(|v_\star|)\in L^0(\mu)\quad\text{ for every }v\in\Theta,
\]
is a pointwise norm on \(\Theta\). Therefore, the space \(\Theta\) is an \(L^0(\mu)\)-normed \(\Phi\)-module.
\begin{remark}\label{rmk:str_Ban_bundle_ex}{\rm
Each strong Banach bundle (in the sense of \cite{LGP22}) is in particular a Banach module bundle, in the sense we are going to describe.
Given a measure space \((Q,\mathcal Q,\mathfrak q)\) and strong Banach bundle \((\B_\star,{\rm T})\) over
\((Q,\mathcal Q,\mathfrak q)\), we can regard (thanks to Example \ref{ex:Ban_as_mod}) each space \(\B_q\) as an
\(L^0(\delta_{(q,{\sf o})})\)-Banach \(L^0(\delta_{(q,{\sf o})})\)-module, where \((\{{\sf o}\},\delta_{\sf o})\)
is the one-point probability space. Recall also from Example \ref{ex:trivial_disint} that \(q\mapsto\delta_{(q,{\sf o})}\)
is a disjoint measure-valued map from \(Q\) to \(Q\times\{{\sf o}\}\) and
\(\mathfrak q\otimes\delta_{\sf o}=\int\delta_{(q,{\sf o})}\,\d\mathfrak q(q)\). Moreover,
\(\R_{\mathfrak q}\subseteq L^0(\mathfrak q)\cong L^0(\mathfrak q\otimes\delta_{\sf o})\cong L^0(\mathfrak q;L^0(\delta_{(\star,{\sf o})}))\).
Therefore, up to the above identifications, \((\B_\star,{\rm T},\R_{\mathfrak q})\) is a Banach module bundle consistent with \(q\mapsto\delta_{(q,{\sf o})}\).
\fr}\end{remark}
\begin{definition}[Integral of a Banach module bundle]\label{def:integral_of_bundle}
Let \((Q,\mathcal Q,\mathfrak q)\) be a measure space, \((\X,\Sigma)\) a measurable space, and \(q\mapsto\mu_q\)
a disjoint measure-valued map from \(Q\) to \(\X\). Let \((\mathscr M_\star,\Theta,\Phi)\) be a Banach module
bundle consistent with \(q\mapsto\mu_q\). Then we define \(\int\mathscr M_q\,\d\mathfrak q(q)\) as
\[
\int\mathscr M_q\,\d\mathfrak q(q)\coloneqq\text{completion of the }L^0(\mu)\text{-normed }\Phi\text{-module }\Theta.
\]
The \(L^0(\mu)\)-Banach \(L^0(\mu)\)-module \(\int\mathscr M_q\,\d\mathfrak q(q)\) is referred to as the \textbf{integral}
of \((\mathscr M_\star,\Theta,\Phi)\). Given an element \(v\in\int\mathscr M_q\,\d\mathfrak q(q)\), we will often denote
it by \(\int v_q\,\d\mathfrak q(q)\). Whenever we want to make explicit the choice of the classes \(\Theta\) and \(\Phi\),
we write \(\Gamma(\mathscr M_\star,\Theta,\Phi)\) instead of \(\int\mathscr M_q\,\d\mathfrak q(q)\).
\end{definition}
\begin{theorem}[Identification of the integral]\label{thm:identif_int_strong}
Let \((Q,\mathcal Q,\mathfrak q)\) be a measure space, \((\X,\Sigma)\) a measurable space, and \(q\mapsto\mu_q\)
a disjoint measure-valued map from \(Q\) to \(\X\). Let \((\mathscr M_\star,\Theta,\Phi)\) be a Banach module
bundle consistent with \(q\mapsto\mu_q\). Let us denote by \(\bar\Gamma(\mathscr M_\star,\Theta,\Phi)\)
the space of all those elements \(v\in\mathscr S(\mathscr M_\star)\) satisfying the following properties:
\begin{itemize}
\item[\(\rm i)\)] The element \(|v_\star-w_\star|\in\mathscr S_{\mathfrak q}(L^0(\mu_\star))\) belongs to
\(L^0(\mathfrak q;L^0(\mu_\star))\) for every \(w\in\Theta\).
\item[\(\rm ii)\)] There exists a countable set \(\mathcal C_v\subseteq\Theta\) such that \(v_q\in{\rm cl}_{\mathscr M_q}(\{w_q\,:\,w\in\mathcal C_v\})\) for \(\mathfrak q\)-a.e.\ \(q\in Q\).
\end{itemize}
Then \(\Theta\subseteq\pi_{\mathfrak q}(\bar\Gamma(\mathscr M_\star,\Theta,\Phi))\) and
\(\big(\pi_{\mathfrak q}(\bar\Gamma(\mathscr M_\star,\Theta,\Phi)),|\cdot|\big)\)
is (isomorphic to) the integral \(\int\mathscr M_q\,\d\mathfrak q(q)\).
\end{theorem}
\begin{proof}
Set \(\Gamma\coloneqq\pi_{\mathfrak q}(\bar\Gamma(\mathscr M_\star,\Theta,\Phi))\) for brevity.
The inclusion \(\Theta\subseteq\Gamma\) is immediate. Letting
\[
\mathcal G_{\mathfrak q}(\Theta)\coloneqq\bigg\{\sum_{n\in\N}\1_{E_n}^{\mathfrak q}w^n\;\bigg|\;(E_n)_{n\in\N}\subseteq\mathcal Q
\text{ partition of }Q,\,(w^n)_{n\in\N}\subseteq\Theta\bigg\}\subseteq\mathscr S_{\mathfrak q}(\mathscr M_\star),
\]
we have that \(\mathcal G_{\mathfrak q}(\Theta)\) is a vector subspace of \(\mathscr S_{\mathfrak q}(\mathscr M_\star)\), a module over
\(\mathcal G_{\mathfrak q}(\Phi)\), and a subset of \(\Gamma\). One can also readily check that \(|v_\star-z_\star|\in L^0(\mathfrak q;L^0(\mu_\star))\)
for every \(v\in\Gamma\) and \(z\in\mathcal G_{\mathfrak q}(\Theta)\). We claim that
\begin{equation}\label{eq:identif_int_strong_aux}
\forall v\in\Gamma\quad\exists(z^k)_{k\in\N}\subseteq\mathcal G_{\mathfrak q}(\Theta):\quad{\sf j}_{\mathfrak q}(|v_\star-z^k_\star|)\to 0\,\text{ in }L^0(\mu).
\end{equation}
In order to prove it, notice that for any \(k\in\N\) we can find a partition \((E^k_n)_{n\in\N}\subseteq\mathcal Q\) of \(Q\)
and a sequence \((w^{k,n})_{n\in\N}\subseteq\mathcal C_v\) such that \(z^k\coloneqq\sum_{n\in\N}\1_{E^k_n}^{\mathfrak q}w^{k,n}\in\mathcal G_{\mathfrak q}(\Theta)\) satisfies
\(\sfd_{\mathscr M_q}(v_q,z^k_q)\leq\frac{1}{k}\) for \(\mathfrak q\)-a.e.\ \(q\in Q\).
Thanks to Remark \ref{rmk:wlog_mu_q_fin}, we can assume that \(\mathfrak q\) and all the \(\mu_q\)'s are finite. Then
\[
\sfd_{L^0(\mu)}\big({\sf j}_{\mathfrak q}(|v_\star-z^k_\star|),0\big)=
\int\!\!\!\int|v_q-z^k_q|\wedge 1\,\d\mu_q\,\d\mathfrak q(q)=\int\sfd_{\mathscr M_q}(v_q,z^k_q)\,\d\mathfrak q(q)\leq\frac{\mathfrak q(Q)}{k},
\]
whence the validity of \eqref{eq:identif_int_strong_aux} follows. Having \eqref{eq:identif_int_strong_aux} at
disposal, we can prove the following facts:
\begin{itemize}
\item Thanks to \(\big||v_\star|-|z^k_\star|\big|\leq|v_\star-z^k_\star|\),
we obtain that \(|v_\star|\in L^0(\mathfrak q;L^0(\mu_\star))\) for every \(v\in\Gamma\).
\item If \(v,\tilde v\in\Gamma\) are given, then we can find \((z^k)_{k\in\N},(\tilde z^k)_{k\in\N}\subseteq\mathcal G_{\mathfrak q}(\Theta)\)
such that \({\sf j}_{\mathfrak q}(|v_\star-z^k_\star|)\to 0\) and \({\sf j}_{\mathfrak q}(|\tilde v_\star-\tilde z^k_\star|)\to 0\)
in \(L^0(\mu)\). Since for any \(w\in\Theta\) we have \(z^k+\tilde z^k-w\in\mathcal G_{\mathfrak q}(\Theta)\) and
\[
\big||v_\star+\tilde v_\star-w_\star|-|z^k_\star+\tilde z^k_\star-w_\star|\big|\leq\big|(v_\star+\tilde v_\star)-(z^k_\star+\tilde z^k_\star)\big|
\leq|v_\star-z^k_\star|+|\tilde v_\star-\tilde z^k_\star|,
\]
by letting \(k\to\infty\) we deduce that \(|v_\star+\tilde v_\star-w_\star|\in L^0(\mathfrak q;L^0(\mu_\star))\)
and \(v+\tilde v\in\Gamma\).
\item If \(v\in\Gamma\) and \(f\in L^0(\mu)\) are given, then we can find \((z^k)_{k\in\N}\subseteq\mathcal G_{\mathfrak q}(\Theta)\)
and \((f_k)_{k\in\N}\subseteq\mathcal G_{\mathfrak q}(\Phi)\) such that \({\sf j}_{\mathfrak q}(|v_\star-z^k_\star|)\to 0\) and
\(f_k\to f\) in \(L^0(\mu)\). Since for any given element \(w\in\Theta\) we have that \(f_k\cdot z^k-w\in\mathcal G_{\mathfrak q}(\Theta)\) and
\[\begin{split}
\big||(f\cdot v)_\star-w_\star|-|(f_k\cdot z^k)_\star-w_\star|\big|&\leq\big|(f\cdot v)_\star-(f_k\cdot z^k)_\star\big|\\
&\leq|f-f_k||v_\star|+|f_k||v_\star-z^k_\star|,
\end{split}\]
by letting \(k\to\infty\) we deduce that \(|(f\cdot v)_\star-w_\star|\in L^0(\mathfrak q;L^0(\mu_\star))\) and \(f\cdot v\in\Gamma\).
\end{itemize}
Therefore, we conclude that \(\Gamma\) is an \(L^0(\mu)\)-Banach \(L^0(\mu)\)-module
isomorphic to \(\int\mathscr M_q\,\d\mathfrak q(q)\).
\end{proof}
\begin{remark}{\rm
If $\mathscr{M}_\star=L^0(\mu_\star)$ and $\Theta=\Phi=L^0(\q;L^0(\mu_\star))$, then \(L^0(\q;L^0(\mu_\star))\cong\int\mathscr M_q\,\d\mathfrak q(q)\).
\fr}\end{remark}
\subsection{Disintegration of a module with respect to a disjoint measure-valued map}
The next two theorems assert that, starting from a module $\mathscr{M}$ and a disjoint measure-valued map, we can associate a Banach module bundle such that its integral is isomorphic to $\mathscr{M}$.
We have to assume something more either on $\mathscr{M}$ or on the disjoint measure-valued map. Indeed, in the first theorem $\mathscr{M}$ is assumed to be countably-generated, in the second one the measures $\mu_q$'s are all absolutely continuous with respect to a background boundedly-finite, submodular
outer measure $\nu$. 
Regarding the tools involved in the proof of Theorem \ref{thm:disintegration_countably_generated}, we refer the reader to \cite{DMLP21}, where the concepts of Banach $\B$-bundle $\bf E$ and of $\bar\Gamma(\bf E)$ are presented.
\begin{theorem}[Disintegration of a Banach module]
\label{thm:disintegration_countably_generated}
Let \((\X,\Sigma,\mu)\) be a \(\sigma\)-finite measure space and let \(\mathscr M\) be a countably-generated \(L^0(\mu)\)-Banach \(L^0(\mu)\)-module.
Let \((Q,\mathcal Q,\mathfrak q)\) be a measure space and let \(Q\ni q\mapsto\mu_q\) be a disjoint measure-valued map from \(Q\) to \(\X\) such that
\(\mu=\int\mu_q\,\d\mathfrak q(q)\). Then there exists a Banach module bundle \((\mathscr M_\star,\Theta, \Phi)\) consistent with \(Q\ni q\mapsto\mu_q\), such that
\[
\mathscr M\cong\int\mathscr M_q\,\d\mathfrak q(q).
\]
\end{theorem}

\begin{proof}
  Let $\B$ be a universal separable Banach space. Then \cite[Theorem 4.13]{DMLP21} ensures that there exists a separable Banach $\B$-bundle $\bf E$, such that $\Gamma_\mu({\bf E})\cong \mathscr{M}$, where $\Gamma_\mu({\bf E})$ is the quotient space of the space $\bar\Gamma(\bf E)$ of all measurable sections of $\bf E$, under $\mu$-a.e.\ equivalence relation. Recall that the elements of $\bar\Gamma(\bf E)$ are everywhere defined, so that for any $q\in Q$ we can consider an $L^0(\mu_q)$-Banach $L^0(\mu_q)$-module $\Gamma_{\mu_q}(\bf E)$. Denote the projections on quotient spaces with $\pi_\mu:\bar\Gamma({\bf E})\rightarrow \Gamma_\mu({\bf E})$ and $\pi_q:\bar\Gamma({\bf E})\rightarrow \Gamma_{\mu_q}({\bf E})$, for all $q\in Q$. Define $\mathscr{M}_q:=\Gamma_{\mu_q}({\bf E})$, $q\in Q$ and 
  \[\bar\Theta:=\{Q\ni q\mapsto \pi_q(\bar v)\in\mathscr{M}_{q}\,|\,\bar v\in\bar\Gamma({\bf E})\},\quad \Theta:=\pi_\q(\bar\Theta).\]
  Further, let $\Phi:=L^0(\q;L^0(\mu_\star))$. Then $\mathcal{G}_\q(\Phi)=L^0(\mu)$, $\Theta$ is a vector subspace of $\mathscr{S}_\q(\mathscr{M}_\star)$ and a module over $\Phi$. Denote with $v$ the $\q$-a.e.\ defined equivalence class of the map $q\mapsto \pi_q(\bar v)$, then  
  \[|v_\star|={\sf i}_\q\left(|\pi_\mu(\bar v)|\right)\in L^0(\q;L^0(\mu_\star)), \quad\text{for every }\bar v\in\bar\Gamma({\bf E}),\]
  since $|\pi_\mu(\bar v)|$ and $|\pi_q(\bar v)|$ are $\mu$-a.e. and $\mu_q$-a.e. equivalence classes of the map $\X\ni x\mapsto \|\bar v(x)\|_{\B}\in \R$, respectively. Hence, $(\mathscr{M}_\star,\Theta,\Phi)$ is a Banach module bundle consistent with $q\mapsto\mu_q$. Moreover, we obtained that $\Theta$ is an $L^0(\mu)$-normed $L^0(\mu)$-module.
  Further, we define 
  \begin{equation*}
      {\rm J}(\pi_\mu(\bar v)):=\int \pi_q(\bar v)\,\d\q(q),\quad \text{for every }\bar v\in\bar\Gamma({\bf E)}.
  \end{equation*}
The resulting operator ${\rm J}:\Gamma_\mu({\bf E})\rightarrow \Theta$ is an isomorphism of $L^0(\mu)$-normed $L^0(\mu)$-modules. Let us verify this claim. Take $\bar v,\tilde v\in\bar\Gamma(\bf E)$, such that $\pi_\mu(\bar v)=\pi_\mu(\tilde v)$, which means that $\bar v(x)=\tilde v(x)$ $\mu$-a.e.\ on $\X$. Then, since $\mu=\int\mu_q\,\d\q(q)$, follows that $\bar v(x)=\tilde v(x)$ $\mu_q$-a.e. on $\X$ for $\q$-a.e. $q\in Q$, whence $\pi_q(\bar v)=\pi_q(\tilde v)$ for $\q$-a.e. $q\in Q$ and ${\rm J}(\pi_\mu(\bar v))={\rm J}(\pi_\mu(\tilde v))$, so ${\rm J}$ is well-defined. It is clear that $\rm J$ is $L^0(\mu)$-linear and surjective. Furthermore, 
\begin{equation*}
    |{\rm J}(\pi_\mu(\bar v))|={\sf j}_\q(|{\rm J}(\pi_\mu(\bar v))_\star|)=|\pi_\mu(\bar v)|,\quad \text{for every }\bar v\in\bar \Gamma(\bf E), 
\end{equation*}
whence $\rm J$ preserves pointwise norm and thus it is an isomorphism. Thus, we conclude that $\Theta$ is an $L^0(\mu)$-Banach $L^0(\mu)$ module, which coincides with $\int\mathscr{M}_q\,\d\q(q)$.
\end{proof}

Another result in the same direction, which we will need in Section \ref{sec:integration_vf_BV}, is the following:
\begin{theorem}\label{thm:tech_quotient}
Let \((\X,\sfd)\) be a complete, separable metric space. Let \(\nu\) be a boundedly-finite, submodular
outer measure on \(\X\). Let \((Q,\mathcal Q,\mathfrak q)\) be a measure space. Let \(q\mapsto\mu_q\)
be a disjoint measure-valued map from \(Q\) to \(\X\) with \(\mu_q\ll\nu\) for every \(q\in Q\), and
set \(\mu\coloneqq\int\mu_q\,\d\mathfrak q(q)\ll\nu\). Let \(\mathscr M\) be an \(L^0(\nu)\)-Banach \(L^0(\nu)\)-module.
Fix a generating subalgebra \(\mathscr A\) of \(L^0(\nu)\), and a vector subspace \(\mathscr V\) of \(\mathscr M\)
that is also an \(\mathscr A\)-module and generates \(\mathscr M\). Let us define
\[\begin{split}
\Theta&\coloneqq\big\{[v]_{\mu_\star}\in\mathscr S_{\mathfrak q}(\mathscr M_{\mu_\star})\;\big|\;v\in\mathscr V\big\},\\
\Phi&\coloneqq\big\{[f]_{\mu_\star}\in L^0(\mathfrak q;L^0(\mu_\star))\;\big|\;f\in\mathscr A\big\}.
\end{split}\]
Then \((\mathscr M_{\mu_\star},\Theta,\Phi)\) is a Banach module bundle consistent with \(q\mapsto\mu_q\). Moreover, it holds that
\[
\mathscr M_\mu\cong\int\mathscr M_{\mu_q}\,\d\mathfrak q(q),
\]
the canonical isomorphism \({\rm I}\colon\mathscr M_\mu\to\int\mathscr M_{\mu_q}\,\d\mathfrak q(q)\) being the unique
homomorphism satisfying
\[
{\rm I}([v]_\mu)=\int[v]_{\mu_q}\,\d\mathfrak q(q)\quad\text{ for every }v\in\mathscr V.
\]
\end{theorem}
\begin{proof}
Since \(\pi_\mu(\mathscr A)\) is a generating subalgebra of \(L^0(\mu)\) and \(\pi_\mu(\mathscr A)\subseteq\mathcal G_{\mathfrak q}(\Phi)\),
we deduce that \(\mathcal G_{\mathfrak q}(\Phi)\) is dense in \(L^0(\mu)\). Moreover, \(\Theta\) is a vector subspace of
\(\mathscr S_{\mathfrak q}(\mathscr M_{\mu_\star})\), is a module over \(\Phi\), and
\begin{equation}\label{eq:tech_aux_quotient}
|[v]_{\mu_\star}|={\sf i}_{\mathfrak q}(|[v]_\mu|)\in L^0(\mathfrak q;L^0(\mu_\star))\quad\text{ for every }[v]_{\mu_\star}\in\Theta.
\end{equation}
Therefore, \((\mathscr M_{\mu_\star},\Theta,\Phi)\) is a Banach module bundle consistent with \(q\mapsto\mu_q\).
Moreover, we define
\[
\tilde{\rm I}([v]_\mu)\coloneqq\int[v]_{\mu_q}\,\d\mathfrak q(q)=[v]_{\mu_\star}\in\Theta
\subseteq\int\mathscr M_{\mu_q}\,\d\mathfrak q(q)\quad\text{ for every }v\in\mathscr V.
\]
The resulting operator \(\tilde{\rm I}\colon\pi_\mu(\mathscr V)\to\Theta\) is well-defined, \(\pi_\mu(\mathscr A)\)-linear,
and satisfying
\[
|\tilde{\rm I}([v]_\mu)|={\sf j}_{\mathfrak q}(|[v]_{\mu_\star}|)\overset{\eqref{eq:tech_aux_quotient}}=|[v]_\mu|
\quad\text{ for every }v\in\mathscr V.
\]
Since \(\pi_\mu(\mathscr A)\) generates \(L^0(\mu)\), there exists a unique homomorphism
\({\rm I}\colon\mathscr M_\mu\to\int\mathscr M_{\mu_q}\,\d\mathfrak q(q)\) that extends \(\tilde{\rm I}\)
and preserves the pointwise norm. Given that \(\Theta=\tilde{\rm I}(\pi_\mu(\mathscr V))\subseteq{\rm I}(\mathscr M_\mu)\)
generates \(\int\mathscr M_{\mu_q}\,\d\mathfrak q(q)\), we can finally conclude that \({\rm I}\) is an isomorphism,
thus the statement is proved.
\end{proof}
\section{Applications to vector calculus on \texorpdfstring{\(\sf RCD\)}{RCD} spaces}
\label{sec:applications}
This section deals with applications of the language developed so far to the coarea formula for BV functions and to the 1D-localisation, in both cases in the setting of \(\sf RCD\) spaces.
\subsection{Vector fields on the superlevel sets of a BV function}
\label{sec:integration_vf_BV}
First, a preliminary lemma:
\begin{lemma}\label{lem:property_bdry_superlevel}
Let \((\X,\sfd,\mm)\) be a PI space and \(f\in BV(\X)\). Then it holds that
\begin{equation}\label{eq:property_bdry_superlevel}
\partial^*\{\bar f>t\}\cap(\X_f\setminus J_f)\subseteq\{\bar f=t\}\quad\text{ for every }t\in\R.
\end{equation}
In particular, \(t\mapsto P(\{\bar f>t\},\cdot)|_{\X_f\setminus J_f}\) is a disjoint measure-valued
map from \(\R\) to \(\X\).
\end{lemma}
\begin{proof}
Let \(x\in\partial^*\{\bar f>t\}\cap(\X_f\setminus J_f)\) be given. Since \(\Theta^*(\{\bar f>t\},x)=\Theta^*(\{f>t\},x)>0\), we deduce
that \(t\leq f^\vee(x)\). Since \(\Theta^*(\{\bar f\leq t\},x)=\Theta^*(\{f\leq t\},x)=\Theta^*(\X\setminus\{f>t\},x)>0\), we deduce that
\(t\geq f^\wedge(x)\). Since \(x\in\X_f\setminus J_f\) ensures that \(f^\vee(x)=f^\wedge(x)=\bar f(x)\),
we can conclude that \(t=\bar f(x)\), proving \eqref{eq:property_bdry_superlevel}. It follows that
\begin{equation}\label{eq:property_bdry_superlevel_aux}
P(\{\bar f>t\},\cdot)|_{\X_f\setminus J_f}\quad\text{ is concentrated on }\{\bar f=t\}.
\end{equation}
Now notice that we have \(\{\bar f=t\}\cap\{\bar f=s\}=\varnothing\) whenever \(t,s\in\R\) and \(t\neq s\).
This observation, in combination with Theorem \ref{thm:coarea} and \eqref{eq:property_bdry_superlevel_aux},
implies that \(t\mapsto P(\{\bar f>t\},\cdot)|_{\X_f\setminus J_f}\) is a disjoint measure-valued map.
Therefore, the proof is complete.
\end{proof}
\begin{remark}\label{rmk:def_R_sqcup_N}{\rm
In the next result, as indexing family we will consider the measure space \((Q,\mathcal Q,\mathfrak q)\)
given as follows: \(Q\coloneqq\R\sqcup\N\), we declare that a set \(E\subseteq Q\) belongs to \(\mathcal Q\)
if and only if \(E\cap\R\) is a Borel subset of \(\R\) (in other words, \(\mathcal Q\) is the pushforward of
the Borel \(\sigma\)-algebra of \(\R\) under the inclusion map \(\iota\colon\R\hookrightarrow Q\)), and
\[
\mathfrak q\coloneqq\iota_\#\mathscr L^1+\sum_{n\in\N}\delta_n,
\]
where \(\mathscr L^1\) stands for the one-dimensional Lebesgue measure on \(\R\).
\fr}\end{remark}
\begin{proposition}\label{prop:coarea_disj}
Let \((\X,\sfd,\mm)\) be an isotropic PI space and \(f\in BV(\X)\). Fix a countable Borel partition \((\Gamma_n)_n\)
of \(J_f\) with the following property: for any \(n\in\N\), there is \(t_n\in\R\) such that \(\{\bar f>t_n\}\) is of
finite perimeter and \(\Gamma_n\subseteq\partial^*\{\bar f>t_n\}\). Let \((Q,\mathcal Q,\mathfrak q)\) be
as in Remark \ref{rmk:def_R_sqcup_N} and define
\[
\mu^f_q\coloneqq\left\{\begin{array}{ll}
P(\{\bar f>t\},\cdot)|_{\X_f\setminus J_f}\\
(f^\vee-f^\wedge)P(\{\bar f>t_n\},\cdot)|_{\Gamma_n}
\end{array}\quad\begin{array}{ll}
\text{ if }q=t\text{ for some }t\in\R,\\
\text{ if }q=n\text{ for some }n\in\N.
\end{array}\right.
\]
Then \(q\mapsto\mu^f_q\) is a disjoint measure-valued map from \(Q\) to \(\X\). Moreover, it holds that
\[
|Df|=\int\mu^f_q\,\d\mathfrak q(q).
\]
\end{proposition}
\begin{proof}
It readily follows from Lemma \ref{lem:property_bdry_superlevel} that \(q\mapsto\mu^f_q\) is a disjoint
measure-valued map from \(Q\) to \(\X\). Moreover, using Theorem \ref{thm:coarea} and \eqref{eq:jump_isotr}
we see that for any \(E\subseteq\X\) Borel it holds
\[\begin{split}
|Df|(E)&=|Df|(E\cap(\X_f\setminus J_f))+|Df|(E\cap J_f)\\
&=\int_\R P(\{\bar f>t\},E\cap(\X_f\setminus J_f))\,\d t+\sum_{n\in\N}\int_{E\cap\Gamma_n}(f^\vee-f^\wedge)\,\d P(\{\bar f>t_n\},\cdot)\\
&=\int\mu^f_t(E)\,\d\mathscr L^1(t)+\sum_{n\in\N}\mu^f_n(E)=\int\mu^f_q(E)\,\d\mathfrak q(q).
\end{split}\]
This proves that \(|Df|=\int\mu^f_q\,\d\mathfrak q(q)\), thus accordingly the statement is achieved.
\end{proof}
\begin{theorem}
\label{thm:disintegration_L0Df}
Let \((\X,\sfd,\mm)\) be an \({\sf RCD}(K,N)\) space, with \(K\in\R\) and \(N\in[1,\infty)\). Let \(f\in BV(\X)\) be given.
Let \(q\mapsto\mu^f_q\) be defined as in Proposition \ref{prop:coarea_disj}. Let us also define
\(\big(L^0_{\mu^f_\star}(T\X),\Theta_f,\Phi_f\big)\) as
\[\begin{split}
\Theta_f&\coloneqq\big\{[v]_{\mu^f_\star}\in\mathscr S_{\mathfrak q}\big(L^0_{\mu^f_\star}(T\X)\big)
\;\big|\;v\in{\rm Test}(T\X)\big\},\\
\Phi_f&\coloneqq\big\{[f]_{\mu^f_\star}\in L^0(\mathfrak q;L^0(\mu^f_\star))\;\big|\;f\in{\rm Test}(\X)\big\}.
\end{split}\]
Then \(\big(L^0_{\mu^f_\star}(T\X),\Theta_f,\Phi_f\big)\) is a Banach module bundle consistent with \(q\mapsto\mu^f_q\). Moreover,
\[
L^0_{|Df|}(T\X)\cong\int L^0_{\mu^f_q}(T\X)\,\d\mathfrak q(q),
\]
the isomorphism \({\rm I}_f\colon L^0_{|Df|}(T\X)\to\int L^0_{\mu^f_q}(T\X)\,\d\mathfrak q(q)\) being the unique
homomorphism satisfying
\[
{\rm I}_f\big([v]_{|Df|}\big)=\int[v]_{\mu^q_f}\,\d\mathfrak q(q)\quad\text{ for every }v\in{\rm Test}(T\X).
\]
\end{theorem}
\begin{proof}
It is enough to show that we are in position to apply Theorem \ref{thm:tech_quotient} with suitable choices of the objects involved. Indeed, with the notation therein let $\nu:={\rm Cap}$, $\mu_q:=\mu_q^f$ and notice that $\mu_q^f \ll {\rm Cap}$, as a consequence of \eqref{eq:ac_Cap}.
Let $\mathscr{M}:=L^0_{\rm Cap}(T\X)$, \(\mathscr A\coloneqq{\rm Test}(\X)\), and \(\mathscr V\coloneqq{\rm Test}(T\X)\).
With these choices $\mathscr A$ is a generating subalgebra of $L^0({\rm Cap})$ and $\mathscr V$ is a vector subspace of $\mathscr{M}$ and a module over $\mathscr A$, which generates the $L^0(\rm Cap)$-Banach $L^0(\rm Cap)$-module $\mathscr{M}$.
\end{proof}
\subsection{Vector fields on the integral lines of a gradient}\label{sec:integral_vf_localization}
We recall here the theory of 1D-localisation, referring the reader to \cite{Cav14,CavMon15,CavMil16,CM20}. In particular, for the general construction we follow the presentation of \cite{CM20},
where it is not assumed that $\mm(\X)=1$. Let \((\X,\sfd,\mm)\) be an \({\sf RCD}(K,N)\) space, with \(K\in\R\) and \(N\in[1,\infty)\). Given any
\(1\)-Lipschitz function \(u\colon\X\to\R\), its \textbf{transport set} \(\mathcal T_u\) is defined as the Borel set
\[
\mathcal T_u\coloneqq\big\{x\in\X\;\big|\;|u(x)-u(y)|=\sfd(x,y)\text{ for some }y\in\X\setminus\{x\}\big\}.
\]
The transport set \(\mathcal T_u\) can be written (up to \(\mm\)-null sets) as the union of a family
\(\{\X_\alpha\}_{\alpha\in Q}\) of \textbf{transport rays} (see \cite[Theorem 5.5]{Cav14}). Each \(\X_\alpha\) is isometric to a closed real interval $I_\alpha$,
i.e.\ we can find a surjective isometry \(\gamma_\alpha\colon I_\alpha \to\X_\alpha\).
Given any \(\alpha,\beta\in Q\) with \(\alpha\neq\beta\), we also have that
\[
\gamma_\alpha( \overset{\circ}{I_\alpha})\cap\gamma_\beta(\overset{\circ}{I_\beta})=\varnothing.
\]
Moreover, one can find a \(\sigma\)-algebra \(\mathcal Q\) on \(Q\), a finite measure \(\mathfrak q\geq 0\) on
\((Q,\mathcal Q)\), and a disjoint measure-valued map \(\alpha\mapsto\mm_\alpha\) from \(Q\) to \(\X\) such that
\(\mm_\alpha(\X\setminus\X_\alpha)=0\) for every \(\alpha\in Q\) and
\[
\mm|_{\mathcal T_u}=\int\mm_\alpha\,\d\mathfrak q(\alpha).
\]
In addition, it holds that \(\mm_\alpha=h_\alpha\,\mathcal H^1|_{\X_\alpha}\) for some Borel
function \(h_\alpha\colon I_\alpha \to(0,+\infty)\) (where \(\mathcal H^1\) stands for the one-dimensional Hausdorff measure)
and the space \((\X_\alpha,\sfd,\mm_\alpha)\) is an \({\sf RCD}(K,N)\) space (see \cite[Theorem 4.2]{CavMon15}). In particular, the tangent module \(L^0(T\X_\alpha)\)
can be identified with \(L^0(\mm_\alpha)\) via the differential operator \(\d\gamma_\alpha\colon L^0(\mm_\alpha)\to L^0(T\X_\alpha)\)
of \(\gamma_\alpha\) (in the sense of \cite{GP16}).
\medskip

Given any \(f\colon\X\to\R\) Lipschitz and \(\alpha\in Q\), we define
\begin{equation}
\label{eq:definition_partialtilde}
\tilde\partial_\alpha f(x)\coloneqq\frac{\d}{\d t}(f\circ\gamma_\alpha)(t)\bigg|_{t=\gamma_\alpha^{-1}(x)}
\quad\text{ for }\mathcal H^1\text{-a.e.\ }x\in\X_\alpha.
\end{equation}
The resulting function \(\tilde\partial_\alpha f\colon\X_\alpha\to\R\) is an element of \(L^0(\mm_\alpha)\).
It also follows from the results of \cite[Section 4.1]{CM20} that \(\tilde\partial_\star f\in L^0(\mathfrak q;L^0(\mm_\star))\),
and \cite[Theorem 4.5]{CM20} ensures that
\begin{equation}\label{eq:ptwse_norm_partial_f}
\tilde\partial f\coloneqq{\sf i}_{\mathfrak q}^{-1}(\tilde\partial_\star f)=-\1_{\mathcal T_u}^\mm\langle\nabla f,\nabla u\rangle\in L^0(\mm|_{\mathcal T_u}).
\end{equation}
Indeed, the infinitesimal Hilbertianity assumption gives \(D^-f(-\nabla u)=D^+f(-\nabla u)=-\langle\nabla f,\nabla u\rangle\)
(cf.\ with \cite{Gigli12}). 
From the very definition in \eqref{eq:definition_partialtilde} and the fact that given $t,s \in I_\alpha$ with $s >t$, $u(\gamma_\alpha(t))-u(\gamma_\alpha(s))=|s-t|$, it follows that for $\mathfrak{q}$-a.e.\ $\alpha$, $\tilde{\partial}_\alpha u(x)=1$ for $\mathcal{H}^1$-a.e.\ $x \in \X_\alpha$. Thus
\begin{equation}
    \1_{T_u}^\mm|\nabla u|=|{\sf i}_{\mathfrak q}^{-1}(\tilde\partial_\star u)|=1
\end{equation}
holds \(\mm\)-a.e.\ on \(\mathcal T_u\).
We also define \(\partial_\star f\in\mathscr S_{\mathfrak q}(L^0(T\X_\star))\) as
\[
\partial_\alpha f\coloneqq\d\gamma_\alpha(\tilde\partial_\alpha f)\in L^0(T\X_\alpha)\quad\text{ for }\mathfrak q\text{-a.e.\ }\alpha\in Q.
\]
Given $v \in L^0(T\X)$ and a Borel set $A \subseteq \X$, we denote the Banach module generated by \(v\) on \(A\) as
\begin{equation*}
    \langle v \rangle_A\coloneqq\big\{ \1_A^\mm f\cdot v\;\big|\; f \in L^0(\mm) \big\}.
\end{equation*}
It has the structure of $L^0(\mm)$-Banach $L^0(\mm)$-module (and of $L^0(\mm|_A)$-Banach $L^0(\mm|_A)$-module).
\begin{theorem}\label{thm:disint_needles}
Let \((\X,\sfd,\mm)\) be an \({\sf RCD}(K,N)\) space, with \(K\in\R\) and \(N\in[1,\infty)\). Let \(u\colon\X\to\R\)
be a given \(1\)-Lipschitz function. Let us define \(\Theta^u\subseteq\mathscr S_{\mathfrak q}\big(L^0(T\X_\star)\big)\) as
\[
\Theta^u\coloneqq\bigg\{\sum_{n\in\N}\1_{E_n}^{\mathfrak q}\partial_\star f_n\;\bigg|\;
(E_n)_{n\in\N}\subseteq\mathcal Q\text{ partition of }Q,\,(f_n)_{n\in\N}\subseteq{\rm LIP}(\X)\bigg\}.
\]
Then \(\big(L^0(T\X_\star),\Theta^u,L^0(\mathfrak q;L^0(\mm_\star))\big)\) is a Banach module bundle consistent with \(\alpha\mapsto\mm_\alpha\). Also,
\[
\int L^0(T\X_\alpha)\,\d\mathfrak q(\alpha)\cong\langle\nabla u\rangle_{\mathcal T_u}\subseteq L^0(T\X),
\]
the canonical isomorphism \({\rm J}_u\colon\int L^0(T\X_\alpha)\,\d\mathfrak q(\alpha)\to\langle\nabla u\rangle_{\mathcal T_u}\)
being the unique homomorphism with
\[
{\rm J}_u\bigg(\int\partial_\alpha f\,\d\mathfrak q(\alpha)\bigg)=-\1_{\mathcal T_u}^\mm\langle\nabla f,\nabla u\rangle\nabla u
\quad\text{ for every }f\colon\X\to\R\text{ Lipschitz.}
\]
\end{theorem}
\begin{proof}
Since \(\Theta^u\) is a module over \(L^0(\mathfrak q;L^0(\mm_\star))\) and for any
\(v_\star=\sum_{n\in\N}\1_{E_n}^{\mathfrak q}\partial_\star f_n\in\Theta^u\) we have
\[
|v_\star|=\sum_{n\in\N}\1_{E_n}^{\mathfrak q}|\partial_\star f_n|=\sum_{n\in\N}\1_{E_n}^{\mathfrak q}|\tilde\partial_\star f_n|
\overset{\eqref{eq:ptwse_norm_partial_f}}\in L^0(\mm|_{\mathcal T_u}),
\]
we obtain that \(\big(L^0(T\X_\star),\Theta^u,L^0(\mathfrak q;L^0(\mm_\star))\big)\) is a Banach module bundle consistent with \(\alpha\mapsto\mm_\alpha\).
Now let us define the operator \(\tilde{\rm J}_u\colon\{\partial_\star f\;\big|\;f\in{\rm LIP}(\X)\}\to\langle\nabla u\rangle_{\mathcal T_u}\) as
\[
\tilde{\rm J}_u(\partial_\star f)\coloneqq-\1_{\mathcal T_u}^\mm\langle\nabla f,\nabla u\rangle\nabla u\quad\text{ for every }f\in{\rm LIP}(\X).
\]
Since \(\{\partial_\star f\,:\,f\in{\rm LIP}(\X)\}\) generates \(\int L^0(T\X_\alpha)\,\d\mathfrak q(\alpha)\) and we can estimate
\[
|\tilde{\rm J}_u(\partial_\star f)|=\1_{\mathcal T_u}^\mm|\langle\nabla f,\nabla u\rangle||\nabla u|
\overset{\eqref{eq:ptwse_norm_partial_f}}=|\partial_\star f|\quad\text{ for every }f\in{\rm LIP}(\X),
\]
the map \(\tilde{\rm J}_u\) can be uniquely extended to a homomorphism
\({\rm J}_u\colon\int L^0(T\X_\alpha)\,\d\mathfrak q(\alpha)\to\langle\nabla u\rangle_{\mathcal T_u}\)
that preserves the pointwise norm. Finally, let us check that \({\rm J}_u\) is surjective. Fix an arbitrary
element of \(\langle\nabla u\rangle_{\mathcal T_u}\), which can be written (uniquely) as \(h\nabla u\)
for some \(h\in L^0(\mm|_{\mathcal T_u})\). Hence, we have that
\[
{\rm J}_u(-h\cdot\partial_\star u)=-h\cdot{\rm J}_u(\partial_\star u)=
h\1_{\mathcal T_u}^\mm\langle\nabla u,\nabla u\rangle\nabla u=h\nabla u,
\]
thus proving the surjectivity of \({\rm J}_u\). Consequently, the statement is achieved.
\end{proof}
\begin{remark}{\rm
We point out that the above results hold in the more general setting of essentially non-branching, infinitesimally
Hilbertian \({\sf MCP}(K,N)\) spaces.
\fr}\end{remark}
\subsection{Unit normals and gradient lines}
In this conclusive section, we combine the results of Sections \ref{sec:integration_vf_BV} and \ref{sec:integral_vf_localization}.
First, we recall the following result, which follows from \cite[Theorem 2.4]{bru2019rectifiability}:
\begin{theorem}[Gauss--Green]\label{thm:Gauss-Green}
Let \((\X,\sfd,\mm)\) be an \({\sf RCD}(K,N)\) space. Let \(E\subseteq\X\) be a set of finite perimeter with \(\mm(E)<+\infty\).
Then there exists a unique element \(\nu_E\in L^0_{P(E,\cdot)}(T\X)\) such that
\[
\int_E{\rm div}(v)\,\d\mm=-\int\langle[v]_{P(E,\cdot)},\nu_E\rangle\,\d P(E,\cdot)\quad\text{ for every }v\in{\rm Test}(T\X).
\]
The identity \(|\nu_E|=1\) holds \(P(E,\cdot)\)-a.e.\ on \(\X\). Moreover, is \(\varphi\colon\X\to\R\) is bounded Lipschitz, then
\[
\int_E{\rm div}(\varphi v)\,\d\mm=-\int\varphi\langle[v]_{P(E,\cdot)},\nu_E\rangle\,\d P(E,\cdot)\quad\text{ for every }v\in{\rm Test}(T\X).
\]
\end{theorem}
\begin{remark}\label{rmk:when_TV_m}{\rm
Let \((\X,\sfd,\mm)\) be an \({\sf RCD}(K,N)\) space with \(\mm(\X)<+\infty\) and let \(u\colon\X\to\R\) be a
Lipschitz function. Then we know from \cite[Remark 3.5]{GigliHan14} that \(u\in BV(\X)\) and \(|Du|={\rm lip}(u)\mm\).
In particular, if \(\mathcal T_u=\X\) up to \(\mm\)-null sets, then \(|Du|=\mm\) and thus
\(\mm=\int_\R P(\{u>t\},\cdot)\,\d t\). Notice also that, since \(u\) is continuous, it holds that \(\bar u=u\),
so that \(\X_u=\X\) and \(J_u=\varnothing\). Therefore, we have that \(t\mapsto\mu^u_t\coloneqq P(\{u>t\},\cdot)\)
is a disjoint measure-valued map from \(\R\) to \(\X\) and Theorem \ref{thm:disintegration_L0Df} provides an
isomorphism of \(L^0(\mm)\)-Banach \(L^0(\mm)\)-modules \({\rm I}_u\colon L^0(T\X)\to\int L^0_{\mu^u_t}(T\X)\,\d t\).
\fr}\end{remark}

Now we are in a position to state and prove the main result of this section:
\begin{theorem}\label{thm:unit_normal_vs_needles}
Let \((\X,\sfd,\mm)\) be an \({\sf RCD}(K,N)\) space with \(\mm(\X)<+\infty\).
Fix a \(1\)-Lipschitz function \(u\colon\X\to[0,+\infty)\) such that \(\mathcal T_u=\X\) up to \(\mm\)-negligible sets.
Let \({\rm I}_u\colon L^0(T\X)\to\int L^0_{\mu^u_t}(T\X)\,\d t\) and
\({\rm J}_u\colon\int L^0(T\X_\alpha)\,\d\mathfrak q(\alpha)\to\langle\nabla u\rangle_\X\)
be as in Remark \ref{rmk:when_TV_m} and Theorem \ref{thm:disint_needles}, respectively. Then
\[
({\rm I}_u\circ{\rm J}_u)\bigg(\int\partial_\alpha u\,\d\mathfrak q(\alpha)\bigg)
=-\int_\R\nu_{\{u>t\}}\,\d t\in\int L^0_{\mu^u_t}(T\X)\,\d t.
\]
\end{theorem}
\begin{proof}
For brevity, we denote \(\partial_\star u=\int\partial_\alpha u\,\d\mathfrak q(\alpha)\) and
\(V\coloneqq({\rm I}_u\circ{\rm J}_u)(\partial_\star u)\in\int L^0_{\mu^u_t}(T\X)\,\d t\).
Fix a countable dense subset \(\mathcal C\) of \({\rm Test}(\X)\). Notice that the family \(\mathcal F\)
of those elements of \({\rm Test}(T\X)\) of the form \(\sum_{i=1}^n g_i\bar\nabla f_i\) for some \(f_i,g_i\in\mathcal C\)
generates \(L^0_{\rm Cap}(T\X)\). If \(\varphi\in{\rm LIP}_{bs}(\R)\) and \(f,g\in\mathcal C\), then
\[\begin{split}
\int_0^{+\infty}\varphi(t)\int g\langle[\nabla f]_{\mu^u_t},V_t\rangle\,\d P(\{u>t\},\cdot)\,\d t&=
\int_0^{+\infty}\!\!\int\big[\langle\varphi\circ u\,g\nabla f,{\rm J}_u(\partial_\star u)\rangle\big]_{\mu^u_t}\,\d P(\{u>t\},\cdot)\,\d t\\
&=-\int\langle\varphi\circ u\,g\nabla f,\nabla u\rangle\,\d\mm=\int u\,{\rm div}(\varphi\circ u\,g\nabla f)\,\d\mm\\
&=\int_0^{+\infty}\!\!\int_{\{u>t\}}{\rm div}(\varphi\circ u\,g\nabla f)\,\d\mm\,\d t\\
&=-\int_0^{+\infty}\!\!\int\varphi\circ u\,g\langle[\nabla f]_{\mu^u_t},\nu_{\{u>t\}}\rangle\,\d P(\{u>t\},\cdot)\,\d t\\
&=-\int_0^{+\infty}\varphi(t)\int g\langle[\nabla f]_{\mu^u_t},\nu_{\{u>t\}}\rangle\,\d P(\{u>t\},\cdot)\,\d t,
\end{split}\]
where we used Remark \ref{rmk:when_TV_m}, the layer cake representation formula, Theorem \ref{thm:Gauss-Green},
and the fact that the measure \(\mu^u_t=P(\{u>t\},\cdot)\) is concentrated on \(\partial\{u>t\}\subseteq\{u=t\}\).
Thanks to the arbitrariness of \(\varphi\in{\rm LIP}_{bs}(\R)\), we deduce that for every \(f,g\in\mathcal C\) it holds that
\[
\int\langle g[\nabla f]_t,V_t\rangle\,\d P(\{u>t\},\cdot)=-\int\langle g[\nabla f]_t,\nu_{\{u>t\}}\rangle\,\d P(\{u>t\},\cdot)
\quad\text{ for }\mathcal L^1\text{-a.e.\ }t\in\R.
\]
Since \(\mathcal F\) is countable, we can find a \(\mathscr L^1\)-null set \(N\subseteq\R\) such that
\(\int\langle w,V_t\rangle\,\d\mu^u_t=-\int\langle w,\nu_{\{u>t\}}\rangle\,\d\mu^u_t\) holds
for every \(t\in\R\setminus N\) and \(w\in\mathcal F\). It follows that \(V_t=-\nu_{\{u>t\}}\) for
\(\mathcal L^1\)-a.e.\ \(t\in\R\), so that accordingly \(({\rm I}_u\circ{\rm J}_u)(\partial_\star u)=V=-\nu_{\{u>\star\}}\).
Therefore, the proof of the statement is complete.
\end{proof}

Observe that Theorem \ref{thm:unit_normal_vs_needles} applies, for example, to \(u\coloneqq\sfd(\cdot,\bar x)\)
for any point \(\bar x\in\X\).
\begin{remark}{\rm
We point out that Theorem \ref{thm:disintegration_L0Df} can be generalised to functions of \emph{local bounded variation},
thus in turn the finiteness assumption on \(\mm\) in Theorem \ref{thm:unit_normal_vs_needles} can be dropped. However, we
do not pursue this goal here, since the results are already quite technically demanding.
\fr}\end{remark}
\def\cprime{$'$} \def\cprime{$'$}

\end{document}